\documentclass[reqno,11pt]{amsart}

\usepackage{amssymb,graphics,graphicx,wrapfig}
\usepackage{amssymb,epsfig}
\usepackage[usenames,dvipsnames]{color}
\usepackage{hyperref} 
\usepackage{tikz}

\def\be{\begin{equation}}
\def\ee{\end{equation}}
\def\ba{\begin{align}}
\def\ea{\end{align}}
\def\bsplit{\begin{split}}
\def\esplit{\end{split}}
\def\bm{\begin{multline}}
\def\eem{\end{multline}}
\def\bfig{\begin{figure}[htb]}
\def\efig{\end{figure}}
\makeatletter
  \def\tagform@#1{\maketag@@@{\footnotesize{(#1)}\@@italiccorr}}
\makeatother

\renewcommand{\eqref}[1]{(\ref{#1})}

\numberwithin{equation}{section}
\newtheorem{theorem}{Theorem}[section]
\newtheorem{proposition}[theorem]{Proposition}
\newtheorem{lemma}[theorem]{Lemma}

\newtheorem{assumption}{Assumption}

\theoremstyle{remark}
\newtheorem*{remark}{Remark}
 
\renewcommand{\leq}{\;\leqslant\;}
\renewcommand{\geq}{\;\geqslant\;}
\newcommand{\dd}{{\rm d}}
\newcommand{\e}[1]{\,{\rm e}^{#1}\,}
\newcommand{\ii}{{\rm i}}
\newcommand{\sumtwo}[2]{\sum_{\substack{#1 \\ #2}}}

\newcommand{\limtwo}[2]{\lim_{\substack{#1 \\ #2}}}
\newcommand{\limsuptwo}[2]{\limsup_{\substack{#1 \\ #2}}}

\newcommand{\todi}{\; \substack{{\rm d} \\ \longrightarrow \\ \phantom{d}} \;}

\def\eps{\varepsilon}

\newcommand{\caH}{{\mathcal H}}\newcommand{\caN}{{\mathcal N}}\newcommand{\caS}{{\mathcal S}}

\newcommand{\bbE}{{\mathbb E}}\newcommand{\bbN}{{\mathbb N}}\newcommand{\bbP}{{\mathbb P}}\newcommand{\bbR}{{\mathbb R}}\newcommand{\bbZ}{{\mathbb Z}}

\newcommand{\bsa}{{\boldsymbol a}}\newcommand{\bsr}{{\boldsymbol r}}

\newcommand{\bseta}{{\boldsymbol\eta}}

\newcommand{\bslambda}{{\boldsymbol\lambda}}

\newcommand{\bstheta}{{\boldsymbol\theta}}

\begin{document}

{\hfill\small Electron. J. Probab. 19 (2014), no. 82, 1--37. DOI: 10.1214/EJP.v19-3244} \vspace{5mm}

\title[Random partitions in statistical mechanics]{Random partitions in statistical mechanics}

\author[Ercolani, Jansen, Ueltschi]{Nicholas M.\ Ercolani, Sabine Jansen, Daniel Ueltschi}

%\iffalse
\address{Nicholas M.\ Ercolani \hfill\newline
\indent Department of Mathematics, The University of Arizona \hfill\newline
\indent 617 N.\ Santa Rita Ave., P.O.\ Box 210089 \hfill\newline
\indent Tucson, AZ 85721--0089, USA \hfill\newline
{\small\rm\indent http://math.arizona.edu/$\sim$ercolani/}
}
\email{ercolani@math.arizona.edu}

\address{Sabine Jansen \hfill\newline
\indent Ruhr-Universit{\"a}t Bochum \hfill\newline
\indent Universit\"atsstr.\ 150, 44780 Bochum, Germany %\hfill\newline
%{\small\rm\indent http://}
}
\email{sabine.jansen@ruhr-uni-bochum.de}

\address{Daniel Ueltschi \hfill\newline
\indent Department of Mathematics, University of Warwick \hfill\newline
\indent Coventry, CV4 7AL, United Kingdom \hfill\newline
{\small\rm\indent http://www.ueltschi.org}
}
\email{daniel@ueltschi.org}
%\fi

\begin{abstract}
We consider a family of distributions on spatial random partitions that provide a coupling between different models of interest: the ideal Bose gas; the zero-range process; particle clustering; and spatial permutations. These distributions are invariant for a ``chain of Chinese restaurants'' stochastic  process. We obtain results for the distribution of the size of the largest component.
\end{abstract}

\maketitle

%\noindent

%\noindent
{\footnotesize 2010 Math.\ Subj.\ Class.: 60F05, 60K35, 82B05.}

%{\footnotesize\it Dedicated to William G.\ Faris on the occasion of his retirement.}

\tableofcontents

\section{Introduction}

We study systems of random integer partitions that are independent except for a global constraint on their total mass. Such a setting appears, directly or indirectly, in diverse systems of statistical mechanics: the ideal quantum Bose gas, the zero-range process, particle clustering, certain coagulation-fragmentations processes, and some models of spatial permutations. A common feature is the possibility of a Bose--Einstein condensation; namely, under some conditions, a phase transition takes place that is accompanied with the formation of large objects. In the language of probability, a single random variable realizes the large deviation required to satisfy the constraint on its sum; this behavior is a well-known hallmark for sums of heavy-tailed random variables, and in fact many of our results can be read as abstract results for (conditioned) sums of independent random variables.

We introduce the setting in Section \ref{sec spart}. The random objects are ``spatial partitions'', that is, collections of integer partitions indexed by locations. The distribution has a product structure subject to a global constraint. Two marginals play an important r\^ole. The first marginal deals with the site occupation numbers; the resulting distribution is that of the ideal Bose gas or of the zero-range process. The second marginal deals with integer partitions; the resulting distribution is that of the particle clustering and of spatial permutations. The present study originated in an attempt at unifying the latter two systems, and the links with the former systems were rather unexpected. It is useful to establish connections since many results and properties of one system can be transferred to the others.

The special cases are described in Section \ref{sec:examples}. The ideal Bose gas can be found in Section \ref{sec ideal Bose}, the zero-range process in Section \ref{sec zero-range}, particle clustering in Section \ref{sec:droplets}, coagulation-fragmentations processes in Section \ref{sec:becker-doering}, and spatial permutations in Section \ref{sec spatial permutations}.

The measures considered in this article are invariant measures for interesting Markov processes. One process represents customers in a ``chain of Chinese restaurants'' which combines the usual Chinese restaurant process with the zero-range process. It is described in Section \ref{sec Chinese}. This actually holds only when the parameters satisfy certain consistency properties. Another process is a coagulation-fragmentation process which is a variant of the Becker--D\"oring model, see Section \ref{sec:coag-frag}.

The possible occurrence of Bose--Einstein condensation and its consequences are addressed in Section \ref{sec rv}. The relevant asymptotic is the thermodynamic limit of statistical mechanics, and the critical density is given by an explicit formula. We study more specific settings in the last two sections, namely, the case of the trap potential in Section \ref{sec:traps} and the case of the square potential in Section \ref{sec:square}.

\section{Random spatial partitions}
\label{sec spart}

\subsection{Setting}

An \emph{integer partition}\footnote{When there is no risk of confusion with \emph{set partitions}, we will drop the word ``integer'' in front of  ``partition''.} $\lambda$ of the integer $n\in \bbN$ is a finite decreasing sequence $\lambda_1\geq \lambda_2\geq \cdots \geq \lambda_k \geq 1$, of varying length $k$, whose elements add up to $n$: $\sum_{j=1}^k \lambda_j =n$; one often writes ``$\lambda \vdash n$''.  We refer to the length $k$ of the sequence as the \emph{number of components} of the partition $\lambda$. Every partition is uniquely determined by the numbers $r_j(\lambda) = \#\{ i =1,\ldots,k\mid \lambda_i = j\}$, $j\in \bbN$, which count how many times a given integer $j$ appears in the partition $\lambda$; they are often called \emph{occupation numbers} of the partition. We also define the partition of $n=0$ as the empty sequence (length $0$, all occupation numbers equal to $0$).  

We are interested in random integer partitions that have additional structure.  A \emph{spatial partition} of the integer $n$ is a collection $\bslambda = (\lambda_x)_{x\in \bbZ^d}$ of  integer partitions at each site $x \in \bbZ^{d}$. That is, each $\lambda_{x}$ is a $k$-tuple $(\lambda_{x1}, \lambda_{x2}, \dots, \lambda_{xk})$, where $k$ varies and where the integers $\lambda_{xj}$ satisfy
\be
           \lambda_{x1} \geq \lambda_{x2} \geq \dots \geq \lambda_{xk} \geq 1
\ee
(and we allow for the ``empty'' sequence with $k=0$). The spatial partition $\boldsymbol{\lambda}$ is uniquely determined by the 
numbers $r_{xj}$ that count how many times a given integer $j$ appears in the partition at site $x$, 
\be
     	r_{xj} = \#\{ i = 1,2,\dots : \lambda_{xi} = j \}.
\ee
We can form two different types of marginals, summing over integers $j$ or sites $x$; this gives rise to two different types of occupation numbers. We let $\bsr = (r_{j})_{j\geq1}$ denote the sums over all the sites, i.e., $r_{j} = \sum_{x\in\bbZ^{d}} r_{xj}$. The site occupation numbers $\bseta = (\eta_{x})_{x\in\bbZ^{d}}$ are given by
\be
           \eta_{x} = \sum_{i\geq1} \lambda_{xi} = \sum_{j\geq1} j r_{xj}.
\ee
Thus for every $x$, $\lambda_x$ is a partition of the integer $\eta_x$. Sometimes this aspect is stressed and one calls $\bslambda$ a \emph{vector partition} of the vector $\bseta$, written $\bslambda\vdash \bseta$ 
(see e.g.~\cite{vershik96}). 

\begin{figure}[htb]
\begin{picture}(0,0)%
\includegraphics{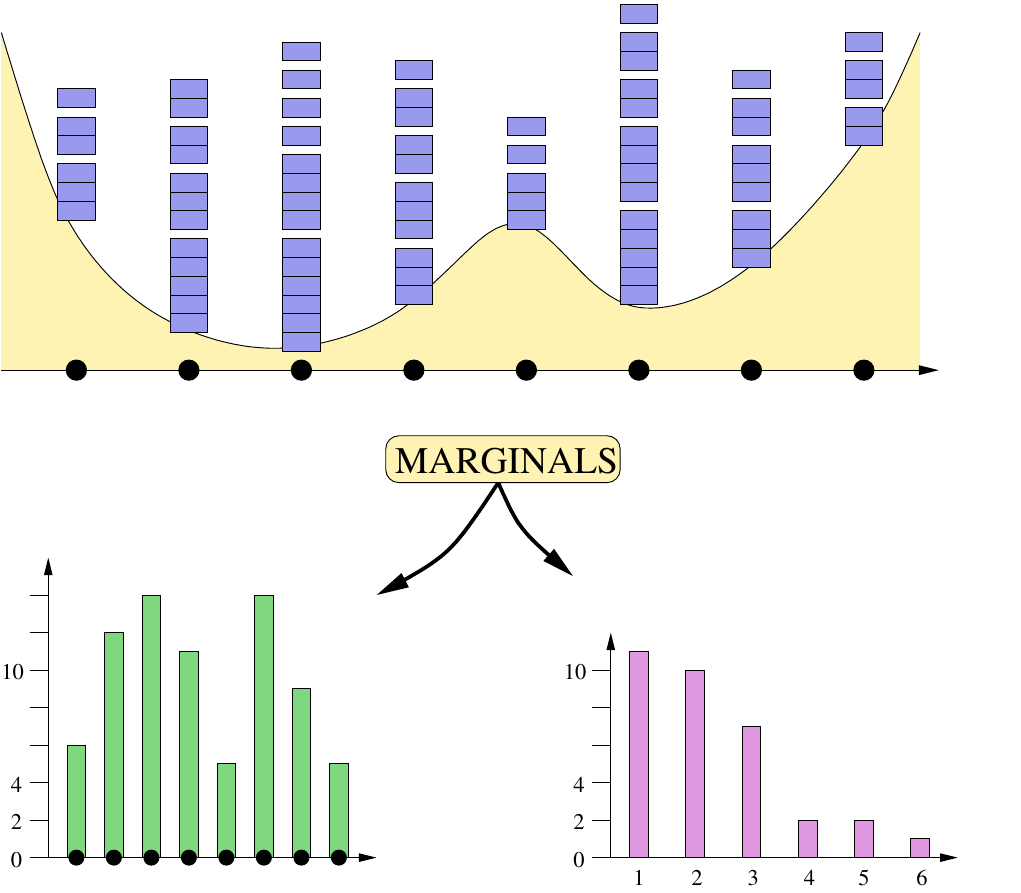}
\end{picture}%
\setlength{\unitlength}{2368sp}%
\begingroup\makeatletter\ifx\SetFigFont\undefined%
\gdef\SetFigFont#1#2#3#4#5{%
  \reset@font\fontsize{#1}{#2pt}%
  \fontfamily{#3}\fontseries{#4}\fontshape{#5}%
  \selectfont}%
\fi\endgroup%
\begin{picture}(8144,7116)(1189,-6640)
\put(4276,-6436){\makebox(0,0)[lb]{\smash{{\SetFigFont{8}{9.6}{\rmdefault}{\mddefault}{\updefault}{\color[rgb]{0,0,0}$\bbZ^d$}%
}}}}
\put(8926,-6436){\makebox(0,0)[lb]{\smash{{\SetFigFont{8}{9.6}{\rmdefault}{\mddefault}{\updefault}{\color[rgb]{0,0,0}$j$}%
}}}}
\put(5926,-4411){\makebox(0,0)[lb]{\smash{{\SetFigFont{8}{9.6}{\rmdefault}{\mddefault}{\updefault}{\color[rgb]{0,0,0}$r_j$}%
}}}}
\put(2926,-1786){\makebox(0,0)[lb]{\smash{{\SetFigFont{7}{8.4}{\rmdefault}{\mddefault}{\updefault}{\color[rgb]{0,0,0}$\lambda_{x1}$}%
}}}}
\put(2926,-1186){\makebox(0,0)[lb]{\smash{{\SetFigFont{7}{8.4}{\rmdefault}{\mddefault}{\updefault}{\color[rgb]{0,0,0}$\lambda_{x2}$}%
}}}}
\put(2926,-736){\makebox(0,0)[lb]{\smash{{\SetFigFont{7}{8.4}{\rmdefault}{\mddefault}{\updefault}{\color[rgb]{0,0,0}$\lambda_{x3}$}%
}}}}
\put(2926,-361){\makebox(0,0)[lb]{\smash{{\SetFigFont{7}{8.4}{\rmdefault}{\mddefault}{\updefault}{\color[rgb]{0,0,0}$\lambda_{x4}$}%
}}}}
\put(8626,164){\makebox(0,0)[lb]{\smash{{\SetFigFont{10}{12.0}{\rmdefault}{\mddefault}{\updefault}{\color[rgb]{0,0,0}$V(x)$}%
}}}}
\put(8776,-2536){\makebox(0,0)[lb]{\smash{{\SetFigFont{10}{12.0}{\rmdefault}{\mddefault}{\updefault}{\color[rgb]{0,0,0}$\bbZ^d$}%
}}}}
\put(4426,-3811){\makebox(0,0)[lb]{\smash{{\SetFigFont{10}{12.0}{\rmdefault}{\mddefault}{\updefault}{\color[rgb]{0,0,0}$\pi_1$}%
}}}}
\put(5551,-3736){\makebox(0,0)[lb]{\smash{{\SetFigFont{10}{12.0}{\rmdefault}{\mddefault}{\updefault}{\color[rgb]{0,0,0}$\pi_2$}%
}}}}
\put(1726,-4036){\makebox(0,0)[lb]{\smash{{\SetFigFont{8}{9.6}{\rmdefault}{\mddefault}{\updefault}{\color[rgb]{0,0,0}$\eta_x$}%
}}}}
\put(2626,-2761){\makebox(0,0)[lb]{\smash{{\SetFigFont{10}{12.0}{\rmdefault}{\mddefault}{\updefault}{\color[rgb]{0,0,0}$x$}%
}}}}
\put(2026,-6586){\makebox(0,0)[lb]{\smash{{\SetFigFont{8}{9.6}{\rmdefault}{\mddefault}{\updefault}{\color[rgb]{0,0,0}$x$}%
}}}}
\end{picture}%
\caption{\small A schematic illustration of spatial partitions and the two relevant marginals.}
\label{fig spatial part}
\end{figure}
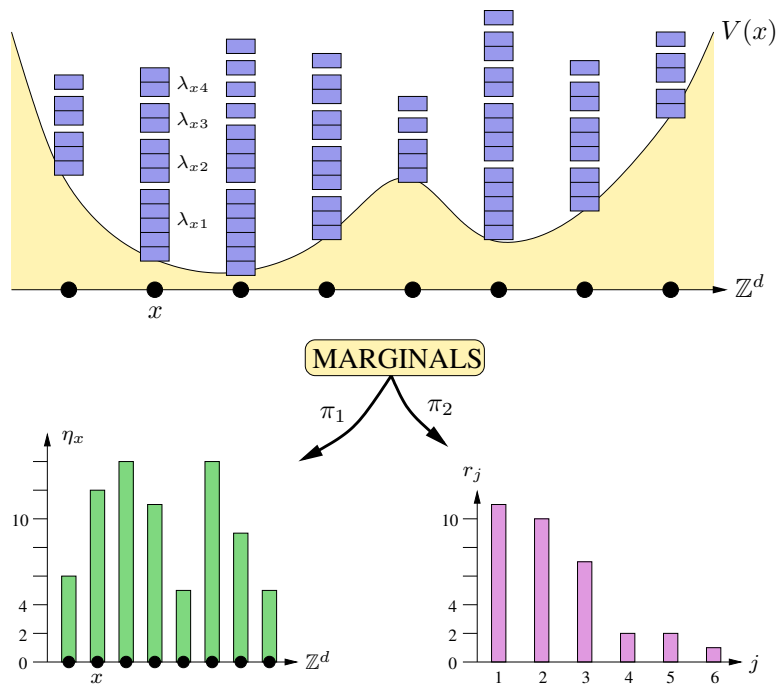

These definitions are illustrated in Figure~\ref{fig spatial part}. The intuition is as follows: we think of $x\in \bbZ^d$ as sites in space (hence the name \emph{spatial} partitions). ``Space'' and ``site" should be taken loosely --- $x$ can be a particle position,  a moment vector $x=k$ in a Fourier transformed picture, or a label for an energy level, see the examples and Table~\ref{table:corr}  in Section~\ref{sec:examples}. Let  $\Lambda_{n}$ denote the set of spatial partitions such that 
\be 
	\sum_{x\in\bbZ^{d}, i\geq1} \lambda_{xi} = \sum_{x\in \bbZ^d,\, j\geq1} j r_{xj} = \sum_{x\in \bbZ^d} \eta_x = \sum_{j\geq 1}  j r_j =n. 
\ee
In the language of statistical mechanics, $\Lambda_n$ is a canonical configuration space with total particle number $n$. Notice that $\Lambda_{n}$ is a countable set. For later purpose we also define $\mathcal{N}_n \subset \bbN_0^{\bbZ}$ as the set of $\bseta$'s with $\sum_{x\in \bbZ^d} \eta_x =n$, and $\mathcal{R}_n \subset \bbN_0^{\bbN}$ as the set of $\bsr$'s with $\sum_{j \in \bbN} r_j = n$. Let $\pi_1: \Lambda_n \to \mathcal{N}_n$, $\bslambda \mapsto \bsr(\bslambda)$ and $\pi_2: \Lambda_n\to \mathcal{R}_n$, $\bslambda \mapsto \bseta(\bslambda)$ be the natural projections. 
% $\Lambda_{n}$ is a countable set and we naturally consider the discrete $\sigma$-algebra.

Apart from the space dimension $d$, the relevant parameters for our probability distribution on $\Lambda_{n}$ are the following:
\begin{itemize}
	\item[(i)] A potential function $V : \bbR^{d} \to (-\infty, \infty]$.
	\item[(ii)] A parameter $\rho \in [0,\infty)$ which represents the density of the system.
		We set $L^{d} = n/\rho$. 
	\item[(iii)] A sequence of non-negative parameters $\bstheta = (\theta_{j})_{j\geq1}$.
\end{itemize}
The probability distribution on $\Lambda_{n}$ is defined as
\be
  	\label{def main prob}
	\bbP_{L,n}(\bslambda) = \frac1{Z_{L,n}} \prod_{x\in\bbZ^{d}} \prod_{j\geq1} \frac1{r_{xj}(\bslambda)!} \biggl( \frac{\theta_{j}}j \e{-j V(x/L)} \biggr)^{r_{xj}(\bslambda)}.
\ee

The number $Z_{L,n}$ is the normalization; it actually depends on $V$ and $\bstheta$, but we alleviate the notation by neglecting to make it explicit. We assume that $Z_{L,n}<\infty$ The relevant asymptotic is the thermodynamic limit where both $n$ and $L$ tend to infinity, with $L$ such that $n = \rho L^{d}$. We will propose different interpretations of the measure $\bbP_{L,n}$ later. One such interpretation is that $\bbP_{L,n}$ is a canonical Gibbs measure for particles moving in the trap potential $V(x/L)$, forming groups at each site. Particles from different groups do not interact; 
 the parameter $\theta_j$ is a Boltzmann weight for intra-group interactions.\footnote{Our interpretation is consistent with the examples given in Section~\ref{sec:examples} but different from that of Vershik~\cite{vershik96} and Pitman~\cite[Chapter 1.5]{pitman}, who consider $Z_{L,n}$ as a microcanonical rather than a canonical partition function.} See Section~\ref{sec:examples} for details.

\subsection{Marginals and conditional probabilities}

An advantage of the probability measure~\eqref{def main prob} is that it allows us to switch between random partitions and sums of independent, {\it infinitely divisible} random variables. The latter play an important r\^ole as stationary measures of the zero-range process. These measures arise as marginals of $\bbP_{L,n}$.

To be sure, sums of independent random variables, corresponding to the occupation numbers $r_j$, have played a significant r\^ole in many recent studies of decomposable random structures (see for instance~\cite{arratia-barbour-tavare});
but, to the best of our knowledge, this connection to the infinitely divisible random variables, corresponding to the site occupation numbers $\eta_x$, has not been noticed before. It allows us to deduce limit laws for random partitions from limit laws for sums of independent variables.

Let  $\bbP_{L,n}\circ \pi_1^{-1}$ and $\bbP_{L,n}\circ \pi_2^{-1}$  be the push-forwards of $\bbP_{L,n}$ under the projections onto $\mathcal{N}_n$ and $\mathcal{R}_n$, respectively. Set 
\be \label{eq:hm} 
            h_{m} = \sum_{\lambda\vdash m} \prod_{j\geq1} \frac1{r_{j}(\lambda)!} \Bigl( \frac{\theta_{j}}j \Bigr)^{r_{j}(\lambda)},
\ee
where the first sum is over the partitions $\lambda$ of the integer $m$. We can think of $h_m$ as the analogue of the normalization $Z_{L,n}$ for a single site (no product over $x$, no background potential $V(x/L)$). Later we will discuss the properties of the map $(\theta_j) \mapsto (h_m)$. 

\begin{proposition}\label{prop:marginals}\hfill
	\begin{itemize} 
	\item[(a)]
	 The measure $\bbP_{L,n}\circ\pi_1^{-1}$ has the product form 	
	\be \label{eq:marginal1} 
		\bbP_{L,n}\bigl( \pi_1^{-1}(\{\bseta\})\bigr) = \bbP_{L,n}\bigl( \{\bslambda:\ \bseta(\bslambda) =\bseta \} \bigr) = \frac{1}{Z_{L,n}} \prod_{x\in \bbZ^d} \bigl( h_{\eta_x} 
		\e{- \eta_x V(x/L) }\bigr). 
	\ee
	\item[(b)] The measure $\bbP_{L,n}\circ \pi_2^{-1}$ is of the Gibbs partition form 
	\begin{multline} \label{eq:marginal-partition} 
		 \bbP_{L,n}\bigl( \pi_2^{-1}(\{\bsr\})\bigr) = \bbP_{L,n}\bigl( \{\bslambda:\ \bsr(\bslambda)  =\bsr \} \bigr) 
		= \frac{1}{Z_{L,n}} \prod_{j \geq 1}\frac{1}{r_j!} \Bigl( \frac{\theta_j}{j} \sum_{x\in \bbZ^d} \e{- j V(x/L)} \Bigr)^{r_j}.  
	\end{multline}
	\end{itemize} 
\end{proposition} 

Proposition~\ref{prop:marginals} has natural probabilistic proof and interpretation, which we defer to Section \ref{sec:poisson}. For now, suffice it to say that the first marginal~\eqref{eq:marginal1} arises as the stationary measure of the zero-range process. 

In order to recover the full measure from the marginals, we give below the conditional measures $\bbP_{L,n}(\bslambda|\bseta(\bslambda) = \bseta)$ and $\bbP_{L,n}(\bslambda|\bsr(\bslambda) =\bsr)$. 
Let $\nu_{m}$ be the measure on integer partitions of $m$ given by 
\be \label{eq:num}
   	\nu_{m}(\lambda) = \frac1{h_{m}} \prod_{j\geq1} \frac1{r_{j}(\lambda)!} \Bigl(   	  \frac{\theta_{j}}j \Bigr)^{r_{j}(\lambda)}.
\ee
The normalization $h_{m}$ is defined in \eqref{eq:hm}.
The measure  $\nu_m$ is the analogue of the measure~\eqref{def main prob} for a single site $x$. It is an example of a \emph{Gibbs random partition} which we describe in more details in Section \ref{sec Gibbs}. The next proposition contains expressions of the conditional probabilities.

\begin{proposition}\label{prop:conditional} 
 	The conditional measures are given by
	\begin{itemize}
		\item[(a)] $\displaystyle \bbP_{L,n}(\bslambda|\bseta(\bslambda) = \bseta) = \prod_{x\in\bbZ^{d}} \nu_{\eta_{x}}(\lambda_{x})$.
		\item[(b)] $\displaystyle \bbP_{L,n}(\bslambda|\bsr(\bslambda) = \bsr) =  \prod_{j\geq1} \biggl( \Bigl( \begin{matrix} r_{j} \\ \{ r_{xj} \}_{x\in\bbZ^{d}} \end{matrix} \Bigr) \prod_{x\in \bbZ^d} p_{xj}^{r_{xj}} \biggr)$, with $\displaystyle p_{xj} = \frac{ \exp( - j V(x/L))}{\sum_{y\in \bbZ^d}  \exp( - j V(y/L))}$.
	\end{itemize}
\end{proposition}
We leave the elementary proof to the reader. Notice that $\bbP_{L,n}(\bslambda|\bseta)$ does not depend on $V$, and $\bbP_{L,n}(\bslambda|\bsr)$ does not depend on $\bstheta$. Part (a) says that, conditioned on the site occupation numbers, the partitions $\lambda_x$ become independent Gibbs partitions. Part (b) says that given $r_j$, the $r_{xj}$'s are multinomial: each of the $r_j$ components of size $j$ chooses the site $x$ with probability $p_{xj}$.

\subsection{Gibbs random partitions}
\label{sec Gibbs}

The measure $\nu_{m}$ defined in \eqref{eq:num} can be viewed as describing Gibbs random partitions, which have been studied in detail before, see \cite{arratia-barbour-tavare, erlihson-granovsky,buv,ercolani-ueltschi,NZ,Zhao,MNZ,NSZ,DM,BM} and Chapters 1.5 and 2.5 in~\cite{pitman}. The main questions deal with the number of elements and their typical size, for given weights $(\theta_{j})$. The probability that the typical size is equal to $\ell$ is defined by
\be \label{eq:typel}
\bbP_{n}(X=\ell) = \bbE_{n}\Bigl(\frac{\ell R_{\ell}}{n}\Bigr) = \frac{\theta_{\ell} \, h_{n-\ell}}{n \, h_{n}}.
\ee
See Section \ref{sec rv} for more discussion about the random variable for the typical size of elements, where the random element is picked with probability that is proportional to its size.

We now study the relation between the $\theta_j$s and the $h_n$s. This is obviously useful in view of the relation above. But this relation is also conceptually interesting since the $\theta_{j}$s are related to the L\'evy measure of a process on $\bbN_{0}$ given by the $h_{n}$s.

Recall that a measure $\mu$ is \emph{infinitely divisible} if for all $n\in \bbN$, there is a measure $\tilde \mu_n$ such that $\mu =\tilde \mu_n*\cdots *\tilde \mu_n$ is the $n$-fold convolution of $\tilde\mu_n$. Similarly we  say that a sequence $(h_m)_{m\geq 0}$ of non-negative numbers is infinitely divisible if for all $n\in \bbN$ there is a sequence of \emph{non-negative} numbers $(h_m^{(n)})$ such that $(h_m)$ is the $n$-fold convolution of $(h_m^{(n)})_{m\geq 0}$, i.e., $h = h^{(n)} *\cdots *h^{(n)}$.
There is a rich theory for infinitely divisible measures~\cite{gnedenko-kolmogorov}, closely related to the topic of L{\'e}vy processes.

\begin{proposition} \label{prop:tjhm}
	Let $(h_m)_{m\geq 0}$ be a sequence of non-negative numbers such that $h_0=1$ and $\sum_m h_mz^m$ has nonzero radius of convergence. 
 Then there is a unique sequence $(\theta_j)_{j\geq 1}$ of real numbers such that Eq.~\eqref{eq:hm} holds for all $m\in \bbN$.  Moreover, we have $\theta_j \geq 0$ for all $j\geq 1$ if and only if  $(h_m)_{m\geq 0}$ is infinitely divisible. 
\end{proposition} 

\begin{proof}
We note the  power series identity 
\be  \label{eq:levy}
	\sum_{m\geq 0} h_m z^m = \exp\biggl( \sum_{j \geq 1} \frac{\theta_j}{j}z^j\biggr)
\ee 
(see e.g.~\cite[Section 3.3]{Aigner}). 
This identity shows that for a given $(h_m)$ there is a unique sequence of real numbers $(\theta_j)$ such that Eq.~\eqref{eq:hm} holds.

Let $R$ denote the radius of convergence of the series $\sum h_{m} z^{m}$. Fix $z\in(0,R)$, and let $C^{(z)}=\sum_{m\geq 0} h_m z^m$ and
$p_m^{(z)}= C^{-1} h_m z^m$. One can check that $(h_m)$ is infinitely divisible if and only if $(p_m)$ is. Furthermore, $(p_m^{(z)})$ defines a probability measure on $\bbN_0$ with cumulant generating function 
\be
	 \log\Bigl( \sum_{m\geq 0} p_m^{(z)} \e{tm}\Bigr) = 
		\sum_{j\geq 1} \frac{\theta_j}{j} z^j \bigl( \e{tj} - 1\bigr). 
\ee
Here, $t$ satisfies $z\e{t}<R$.
We recognize the L{\'e}vy-Khinchin representation for an  infinitely divisible measure in the special case of a measure on $\bbN_0$; it follows from classical results (see e.g.\ Chapter~18 in \cite{gnedenko-kolmogorov}) that $(p_m^{(z)})$ is infinitely divisible if and only if $\theta_j \geq 0$ for all $j\geq 1$. If this is the case, the weights $\alpha_{j}^{(z)} = \frac{\theta_j}{j}z^j$ define a measure on $\bbN$,  the \emph{L{\'e}vy measure} of $(p_m^{(z)})$. 
\end{proof}

Another question of interest is the relation between the asymptotic behaviors of the sequences $(\theta_m)$ and $(h_m)$ as $m\to \infty$. This has been investigated in the references cited at the beginning of this subsection. Relevant results can also be found in the probability literature on the relation between the tails of an infinitely divisible measure and its L{\'e}vy measure, e.g. in  \cite{embrechts-hawkes,embrechts-klueppelberg-mikosch}. 
Here we quote a result of Embrechts and Hawkes~\cite{embrechts-hawkes} on the tail equivalence of an infinitely divisible measure and its L{\'e}vy measure; equivalently, on the relation between the tails of $(h_m)$ and $(\theta_j/j)$. Recall that the convolution between two sequences is defined by $(a*b)_{n} = \sum_{j=0}^{n} a_{j} b_{n-j}$.

\begin{theorem} (Embrechts and Hawkes \cite{embrechts-hawkes})
\label{thm EH}
	Let $(p_n)_{n\geq 0}$ define an infinitely divisible law on $\bbN_0$ with L{\'e}vy measure $(\alpha_j)_{j\geq 1}$. Suppose that $\alpha_j >0$ for all $j\geq 1$. Let $\overline{\alpha}_j = \alpha_j / (\sum_{k\geq 1} \alpha_k)$. The following are equivalent as $n\to \infty$: 
	\begin{itemize}
		\item[(i)] $(p*p)_{n} = 2(1+o(1)) p_n$ and $p_{n+1} /p_{n} \to 1$.
		\item[(ii)] $(\overline{\alpha} * \overline{\alpha})_{n} = 2 (1+o(1)) \overline{\alpha}_n$ and $\overline{\alpha}_{n+1} /\overline{\alpha}_{n} \to 1$.
		\item[(iii)] $p_n = (1+o(1)) \alpha_n$ and $\alpha_{n+1}/\alpha_n \to 1$.
	\end{itemize} 
\end{theorem} 
We note that the L{\'e}vy measure of an integer-valued random variable has always finite mass, hence $\sum_{j\geq 1} \alpha_j <\infty$. 

A probability measure on $\bbN$ satisfying $(p*p)_{n}  \sim 2p_n$ is called \emph{discrete subexponential}. This property suggests that, if two independent variables are conditioned so their sum takes some big value, one variable will take a small value and the other variable will take the appropriate big value. Discrete subexponential variables are necessarily heavy-tailed, that is, $\sum_n p_n z^n$ has radius of convergence $1$. We can apply this theorem if $\sum_j \theta_j /j <\infty$ and $\theta_{j+1}/\theta_{j} \to 1$. Let  $p_m = h_m \exp( - \sum_{j\geq 1} \theta_j/j)$. The condition (ii) becomes
	\be
		\sum_{j=1}^{n-1} \frac{\theta_j}{j} \frac{\theta_{n-j}}{n-j} = 2 \sum_{j=1}^{n/2} \frac{\theta_j}{j} \frac{\theta_{n-j}}{n-j} = 2 \bigl(1+o(1)\bigr) \frac{\theta_n}{n} \sum_{j\geq 1} \frac{\theta_j}{j}. 
	\ee
The first equality is valid if $n$ is even, there is an unimportant correction in the case of $n$ odd.
An immediate consequence of Theorem \ref{thm EH} and of the dominated convergence theorem is the following.

\begin{theorem}
\label{thm buv}
Assume that $\theta_{n+1}/\theta_{n} \to 1$ and that $\theta_{n-j}/\theta_{n} \leq c_{j}$ for $1 \leq j \leq \frac n2$ and all $n$, with $c_{j}$ satisfying $\sum \theta_{j} c_{j} / j < \infty$. Then, as $n\to\infty$,
\[
h_n = \bigl( 1+o(1) \bigr)  \frac{\theta_n}{n} \exp\Bigl(\sum_{j\geq 1} \frac{\theta_j}{j}\Bigr).
\]
\end{theorem}

Notice that $c_{j}$ is necessarily greater or equal to 1, which requires $\sum_j \theta_j /j <\infty$. The theorem was actually proposed in \cite{buv} but the connection with infinitely divisible laws and with Theorem \ref{thm EH} had not been noticed. It applies in particular to stretched exponential weights, $\theta_j/j  = \exp( - j^\gamma)$ with $0<\gamma<1$. 
% \blue{for which $\alpha_j = \theta_j/j$ and $p_n = (1+o(1))\exp( - n^a)$.}  

\subsection{Gibbs partitions and Poisson random variables} \label{sec:poisson}

We conclude this section by explaining in more details the  probabilistic picture behind Proposition \ref{prop:marginals}. To this aim we  generalize a well-known relationship between Gibbs partitions and Poisson variables, see for example  Eq.~(1.53) in~\cite{pitman} or the \emph{conditioning relation} (3.1) in \cite{arratia-barbour-tavare}. 

We assume that there exists $z>0$ such that for all $L>0$,
\be\label{eq:zfinite} 
	\sum_{x\in \bbZ^d} \sum_{j\geq 1} \frac{\theta_j}{j}\, z^j \e{- j V(x/L)} <\infty.
\ee
Let $(\Omega,\mathcal{F},\mathbb{P}_L^z)$ be a probability space and let $(R_{xj})_{x\in \bbZ^d, j \in \bbN}$ be a family of independent Poisson random variables,
\be
	 R_{xj} \sim \text{Poiss}\Bigl( \frac{\theta_j}{j}z^j e^{- j V(x/L) } \Bigr). 
\ee
$\bbP_{L}^{z}$ is the {\it grand-canonical measure}.
The occupation numbers are the random variables
\be
	H_x := \sum_{j\in \bbN} j R_{xj}, \qquad R_j := \sum_{x \in \bbZ^d} R_{xj}. 
\ee
We also define the total number of particles by
\be
N = \sum_{x \in \bbZ^d} \sum_{j\in \bbN} j R_{xj} = \sum_{x \in \bbZ^d} H_{x} = \sum_{j\in \bbN} j R_{j}.
\ee
A moment of thought shows that the law $\bbP_{L,n}$ is recovered by conditioning on the event $\sum_{x,j} j R_{xj} = n$, 
\be \label{eq:pln-cond}
	\bbP_{L,n}(\boldsymbol{\lambda}) = \bbP_L^z \Bigl( \forall x \in \bbZ^d\ \forall j \in \bbN:\ R_{xj} = r_{xj}(\boldsymbol{\lambda})\, \Big|\,   N = n \Bigr),  
\ee
and the normalization satisfies 
\be \label{eq:zln-cond}
	z^n Z_{L,n} = \bbP_L^z \Bigl(  N = n \Bigr)  
		\times \exp \Bigl( \sum_{x\in \bbZ^d} \sum_{j\in \bbN} \frac{\theta_j}{j}\, z^j \e{-j V(x/L)} \Bigr). 
\ee
Note that in Eq.~\eqref{eq:pln-cond} the right-hand side is independent of $z$; this is related to the invariance of the measure $\bbP_{L,n}$ under rescalings $\theta_j \to \theta_j z^j$.  
Under the measure $\bbP_{L}^z$, the random variables $R_j$ are independent Poisson variables 
\be
	R_j \sim \text{Poiss} \Bigl( \frac{\theta_j}{j}\, z^j \sum_{x\in \bbZ^d} \e{- j V(x/L)}\Bigr),  
\ee
and the $H_x$ are independent variables with cumulant generating function 
\be  \label{eq:rjgc} 
	\log \bbE_L^z \bigl[e^{ t H_x}\bigr] = \sum_{j\in \bbN} \frac{\theta_j}{j} z^j \e{- j V(x/L)} \bigl( e^{jt}- 1) \quad (t \in \bbR). 
\ee
Put differently, $H_x$ is an integer-valued, infinitely divisible random variable  with L{\'e}vy measure $\nu(j) = (\theta_j /j) z^j \exp( - j V(x/L))$, compare with the proof of Proposition~\ref{prop:tjhm}.  Furthermore,
\be \label{eq:hetap}
	 \bbP_L^z\Bigl( H_x = m\Bigr)=  h_m z^m \e{- m V(\frac{x}{L})} \e{- \sum_{j\geq 1} 
			\frac{\theta_j}{j} z^j \exp( - j V(x/L) )},
\ee
with $h_m$ as defined in Eq.~\eqref{eq:hm}.  Indeed, 
\be
\begin{aligned} 
	 \bbP_L^z\Bigl( H_x = m \Bigr) &= \sum_{\lambda \vdash m} 
		\prod_{j\geq 1} \bbP_L^z\Bigl( R_{xj} = r_j(\lambda) \Bigr) \\
		& = \sum_{\lambda \vdash m} 
		\prod_{j\geq 1} \left \lbrace \frac{1}{r_j(\lambda)!}\Bigl( \frac{\theta_j}{j}\, z^j \e{-jV(\frac{x}{L}) } \Bigr)^{r_j(\lambda)} \e{- \frac{\theta_j}{j} z^j \exp( - j V(\frac{x}{L}))} \right \rbrace. 
\end{aligned}
\ee

\begin{proof}[Proof of Proposition~\ref{prop:marginals}]
	For (a), we note that
	\be
	\begin{aligned}
		 \bbP_{L,n}\bigl( \pi_1^{-1}(\{\bseta\})\bigr)  & = 
			\bbP_{L}^z \Bigl( \forall x\in \bbZ^d:\ H_x = \eta_x\, \Big|\, 
				N = n \Bigr) 	\\
			& = \frac{1}{\bbP_L^z( \sum_x H_x =n)} \prod_{x\in \bbZ^d} \bbP_L^z \Bigl( H_x = \eta_x \Bigr)
	\end{aligned}
	\ee
	and we conclude with Eqs.~\eqref{eq:hetap} and~\eqref{eq:zln-cond}.  For (b),  we observe that
	\be
	\begin{aligned}
  		\bbP_{L,n}\bigl( \pi_2^{-1}(\{\bsr\})\bigr)
			& = \bbP_{L}^z \Bigl( \forall j \in \bbN:\ R_j = r_j\, \Big|\, 
				N = n \Bigr)\\
			& = \frac{1}{\bbP_L^z(\sum_j j R_j =n)} \prod_{j\in \bbN} \bbP_L^z\bigl(R_j =r_j\bigr) 			
	\end{aligned}
	\ee
	and conclude with Eqs.~\eqref{eq:rjgc} and~\eqref{eq:zln-cond}.
\end{proof}

\section{Relationship with existing models} \label{sec:examples}

Our setting is closely related to several models of interest, namely the ideal Bose gas, the zero-range process, particle clustering, coagulation-fragmentation, Becker--D\"oring, spatial permutations, and population genetics. The relations are explained in this section. Each situation comes with its own language; the keywords and their correspondence are summarized in Table \ref{table:corr}.
\begin{table}[h]
\begin{tabular} {l | lll} 
\hline \hline
Zero-range & particle & site & - \\
\hline
Chinese restaurant & customer & restaurant  & table  \\ 
\hline 
ideal Bose gas,  & particle &  Fourier mode, & cycle \\
spatial permutation & & energy level& \\
\hline
particle clustering,  & particle &  site in space & cluster (droplet) \\
nucleation & & & \\
\hline 
population genetics & individual & colony & same-allele group\\
	&&& within colony \\
\hline \hline
\end{tabular} 
\vspace{.2cm}
\caption{\label{table:corr}
{\footnotesize
Language of the different models and their correspondence.
%$n$ is the total number of customers, $\eta_x(\bslambda)$ is the number of customers at restaurant $x$, $\lambda_x$ is the table size distribution at restaurant $x$, and $r_j(\bslambda)$ counts the total number of tables with $j$ customers, across all restaurants.
}}  
\end{table}

\subsection{Ideal Bose gas}
\label{sec ideal Bose}

Although it was not fully appreciated at the time, Bose--Einstein condensation is the first description of a phase transition in statistical mechanics. The ideal Bose gas is a quantum system whose description involves a complex Hilbert space and the Schr\"odinger equation; but its equilibrium state is a probability distribution of occupation numbers of Fourier modes. It fits our setting, by choosing $V(x) = \beta \|x\|^{2}$ and $\theta_{j} \equiv 1$; $\beta$ is the inverse temperature. With the change of variables $k = \frac{2\pi}L x$, writing
$\lambda_{kj}$ instead of $\lambda_{xj}$, $\bbP_{L,n}$ becomes a probability measure on spatial partitions $(\lambda_{kj})$. Summing over $j$ and writing $\eta_{k}$ instead of $\eta_{x} = \eta_{kL}$, we get the familiar probability measure on occupation numbers
\be \label{eq:ex-bec}
	\bbP_{L,n}(\pi_1^{-1}(\{\bseta\})) = \frac1{Z_{L,n}} \prod_{k \in \frac{2\pi}L \bbZ^{d}} \e{-\beta \|k\|^{2} \eta_{k}}.
\ee
Summing over $k$ we get that the marginal $\bbP_{L,n}\circ \pi_2^{-1}$ is the probability distribution of cycle lengths associated with the ideal Bose gas \cite{Suto2,BCMP}. 
Thus $\bbP_{L,n}$ provides a coupling of the distribution of momenta and cycle lengths for the ideal Bose gas. This generalizes to independent bosons in a trap, when the weight $\exp( - \beta \sum_k \|k\|^2 \eta_k)$ is replaced with $\exp( - \beta \sum_{r\in \mathcal{I}} E_r \eta_r)$, where $\mathcal{I}$ is a countable index set replacing $\bbZ^d$ and $E_r$ ($r\in \mathcal{I}$) are the eigenvalues of the Schr{\"o}dinger operator in the trap. Results in this case were recently obtained in \cite{CD} (see also \cite{BP}).

The interacting Bose gas is much more complicated and it does not fit the present setting. But a partial mean-field approach for the dilute regime suggests that interactions can be approximated by cycle weights $\theta_{j}$ \cite{BU-prl}.

\subsection{Zero-range process}
\label{sec zero-range}

This describes a system of classical particles with stochastic dynamics. There are $n$ particles moving on the sites $\{1,\dots,L\}$ and we let $\bseta \in \caH_{n}$ denote the occupations of the sites. The dynamics is as follows. A particle exits the site $x$ at rate $g(\eta_{x})$, where $g$ is a given function $\bbN \to (0,\infty)$, and it chooses a new site uniformly at random among the neighbors. As it turns out, the spatial dimension does not appear in the stationary measure, so the model is usually studied in $d=1$. The invariant measure is
\be
\bbP_{L,n}(\bseta) = \frac1{Z_{L,n}} \prod_{x=1}^{L} h_{\eta_{x}},
\ee
where the function $h_{k}$ is related to the rates $g$ by
\be
h_{k} = \prod_{i=1}^{k} \frac1{g(i)}.
\ee
It fits the setting studied in this article by choosing the potential $V$ such that $\e{-V(s)} = \chi_{[0,1]}(s)$; see Eq.\ \eqref{eq:marginal1}.

It was noticed by Evans \cite{evans96} that for certain rates $g$, the system possesses a critical density where a sort of Bose--Einstein condensation takes place. See also \cite{GSS,AL,AGL,CG} for further studies.
Variants of the model allow for motion on graphs other than $\bbZ^d$ and hopping mechanisms different from the simple random walk, see~\cite{evans96, waclaw-etal,godreche-luck} and the references therein. When $\e{-V}$ differs from the characteristic function, the marginal~\eqref{eq:marginal1} is the stationary measure of an \emph{inhomogeneous} zero-range process, compare with Eq.~(2.4) of~\cite{godreche-luck}. 

\subsection{Particle clustering}
\label{sec:droplets}
The measure $\bbP_{L,n}$ describes approximately the droplet size distributions for a system of interacting particles in the canonical Gibbs ensemble, see \cite{sator,jk} and the references therein. Let $v$  be a pair potential with a finite range $R>0$. Thus particles at mutual distance larger than $R$ do not interact. 
 With each configuration $(x_1,\ldots,x_n)\in [0,L]^{dn}$ we associate a graph by drawing a line between $x_i,x_j$ whenever $|x_i-x_j|\leq R$, and we call $N_k(\boldsymbol{x})$ the number of connected components having $k$ particles. We have $\sum_{k=1}^N k N_k(\boldsymbol{x}) = n$. We put on $[0,L]^{dn}$ the canonical Gibbs measure at
 inverse temperature $\beta>0$.  If we neglect the constraint that particles belonging to different connected components have mutual distance larger than $R$, the probability of seeing a given droplet size distribution $(N_k)$ becomes approximately
\be \label{eq:ideal-mixture}
	 \prod_{k=1}^n \frac{1}{N_k!} \Bigl( L^d Z_k(\beta) \Bigr)^{N_k},
\ee
where $Z_j(\beta)$ is a partition function over droplet-internal degrees of freedom~\cite{jk}.
Eq.~\eqref{eq:ideal-mixture} corresponds to a measure of the form~\eqref{eq:marginal-partition}  with $\sum_{x} \exp( - V(x/L))$ replaced by $L^d$, and $\theta_j / j= Z_j(\beta)$.

\subsection{Coagulation-fragmentation processes and Becker--D{\"o}ring equations}  \label{sec:becker-doering}
The Becker--D{\"o}ring system of coupled ordinary differential equations~\cite{becker-doering}  is a popular model for the dynamics of nucleation, with interesting long-time behavior~\cite{ball-carr-penrose}. A natural stochastic variant of the model is the following.
%  They are a system of coupled ordinary differential equations for an infinite vector $(\rho_k(t))_{k\in \bbN}$, where $\rho_k(t)$ is the number of $k$-particle droplets per unit volume. The physical picture underlying the model is that individual 
%
Let $(a_j)_{j\geq 1}$ and $(b_j)_{j\geq 1}$ be positive numbers such that  for all $j$, $\frac{\theta_{1} \theta_{j} a_{j}}j = \frac{\theta_{j+1} b_{j+1}}{j+1}$. 
We define a  continuous-time Markov chain with state space $\{\lambda:\, \lambda \vdash n\}$  and two types of transition: 
\begin{itemize}
	\item \emph{Coagulation}: a monomer decides to join a $j$-cluster. The transition 
	$(r_1, r_j, r_{j+1}) \to (r_1-1,r_{j-1}-1, r_{j+1} +1)$ happens at rate $a_j r_1 r_j /L^d$ 
	if $j \neq 1$, and $a_1 r_1(r_1-1) /L^2$ for $j =1$. 
	\item \emph{Fragmentation}: a monomer decides to depart from a $j$-cluster, resulting in the transition $(r_1, r_{j-1},r_j)\to (r_1+1,r_{j-1}+1,r_j -1)$. This happens at rate $b_{j} r_{j}$. 
	\end{itemize} 
One can check through detailed balance that the marginal~\eqref{eq:marginal-partition} with the square potential $\exp( - V(s) ) = \mathbf{1}_{[0,1]}(s)$ is a stationary measure of this process. The dynamics is a stochastic version of the Becker--D{\"o}ring equations much in the same way as the Marcus--Lushnikov coalescent is a stochastic version of the Smoluchowski coagulation equations~\cite{aldous99review}. The model can be easily generalized to allow for joining and departure of groups of size larger than one, and falls into the class of  well-studied coagulation-fragmentation processes, see e.g.\ \cite{durrett-granovsky-gueron, Bertoin}. In Section~\ref{sec:coag-frag} we propose a ``spatial''  version of the process for which our measure $\bbP_{L,n}$ is stationary.

\subsection{Spatial permutations}
\label{sec spatial permutations}

Models of spatial permutations are motivated by Feynman's approach to the Bose gas and by S\"ut\H o's work on cycles \cite{Suto2}. A more general framework was proposed in \cite{BU1}, which was studied further in \cite{BU2,BZ}. Spatial permutations involve a distribution jointly over points in $\bbR^{d}$ and over permutations of these points, with penalties that discourage long jumps. More precisely, the probability space is $\Lambda^{n} \times \caS_{n}$, where $\Lambda$ is the box $[0,L]^{d}$ (with periodic boundary conditions), $n$ is the number of points, and $\caS_{n}$ is the group of permutations of $n$ elements. Let $\xi : \bbR^{d} \to [0,\infty]$ and define
\be
Z_{L,n} = \sum_{\sigma \in \caS_{n}} \int_{\Lambda^{n}} \dd x_{1} \dots \dd x_{n} \prod_{j=1}^{n} \e{-\xi(x_{j}-x_{\sigma(j)})}.
\ee
The probability element for having points $x_{1},\dots,x_{n}$ and permutation $\sigma$ is then
\be
\frac1{Z_{L,n}} \Bigl( \prod_{j=1}^{n} \e{-\xi(x_{j}-x_{\sigma(j)})} \Bigr) \dd x_{1} \dots \dd x_{n}.
\ee
The marginal obtained after summing over permutations is a permanental point process, which we do not discuss here despite interest of its own. We rather focus on properties of permutation cycles. We make the extra assumption that $\e{-\xi}$ has non-negative Fourier transform, which we write $\e{-V}$, with $V$ a real function on $\bbR^{d}$. Then the marginal over the cycle lengths, after integrating over positions and summing over compatible permutations, is precisely the marginal \eqref{eq:marginal-partition}. This follows from the Poisson summation formula, see  \cite{BU1} for details. Notice that the special case $\xi(x) = \beta \|x\|^{2}$ corresponds to the homogeneous ideal Bose gas described in Section \ref{sec ideal Bose}.

\subsection{Population genetics} 
Consider $n$ individuals that carry different alleles of a given gene, and live in different locations or \emph{colonies} $x$. Inside each colony, we group individuals that carry the same allele. This gives rise to a spatial partition. In the Ewens case $\theta_j\equiv\theta$, one might imagine that our measure is stationary for a population model that includes migration (with some colonies possibly more attractive than others) and mutation. Remember that the Ewens sampling formula appears
naturally in the infinite alleles model with mutation rate $\theta$; see~\cite{durrett} and references therein for more background.

\section{Stochastic processes}
\label{sec processes}

Here we propose two continuous-time Markov processes with state space $\Lambda_n$ that have $\bbP_{L,n}$ as a reversible (hence stationary) measure. 
The first process combines the zero-range process with the Chinese restaurant process; it has the nice structural property that the vector of site occupation numbers $\bseta(\bslambda(t))$ is Markovian (regardless of the starting point $\bslambda(0)$) and evolves as the zero-range process. However  we are able to define the Chinese restaurant part only when the weights $\theta_j$ are $j$-independent. For non-constant $\theta_j$'s, we replace the Chinese restaurant step by instant reshuffling.  The second process is a coagulation-fragmentation process that is very natural from the point of view of the Becker--D{\"o}ring model of nucleation explained in Section~\ref{sec:becker-doering}.

\subsection{Chain of Chinese restaurants}
\label{sec Chinese}

Recall that, in the Chinese restaurant process, customers enter the restaurant one by one. The $(n+1)$th customer sits next to an existing customer with probability $\frac1{n+\theta}$, and starts a new table with probability $\frac\theta{n+\theta}$. The table occupation is a random partition and the distribution after $n$ customers is given by the Ewens measure. That is, take $\theta_{j} \equiv \theta$ in Eq.\ \eqref{eq:num}. We refer to Chapter 3 in \cite{pitman} for background and details.

We adapt the dynamics to spatial partitions as follows. To each site $x \in \bbZ^{d}$ is associated a restaurant. The total number of customers is $n$ and it is conserved. Let $\eta_{x}$ denote the number of customers at site $x$, and $\lambda_{x} \vdash \eta_{x}$ denote the table occupation. We consider a continuous-time Markov dynamics where
\begin{itemize}
\item A customer exits restaurant $x$ at rate $g(\eta_{x})$; he is chosen uniformly among the $\eta_{x}$ customers at $x$.
\item The new restaurant $y$ is chosen with probability $t(x,y)$. It may depend on $L$. We assume that $\sum_{y} t(x,y) =1$, and that
\be 
	 \label{eq:pbalance} 
			\e{- V(x/L)} t(x,y) = \e{- V(y/L)} t(y,x).
\ee
\item In the restaurant $y$, the new customer sits next to another customer with probability $\frac1{\eta_{y}+\theta}$, and at an empty table with probability $\frac\theta{\eta_{y}+\theta}$.
\end{itemize}
With $h_{n}$ defined in \eqref{eq:hm}, we take the rate to be $g(n) = h_{n-1} / h_{n}$ (with $g(0)=0$). One can check that $h_{n} = \theta (\theta+1) \dots (\theta+n-1) / n!$ and that $g(n) = \frac n{\theta+n-1}$. A possible choice for $t(x,y)$ is $\e{-V(y/L)} / \sum_{z} \e{-V(z/L)}$ (with $y=x$ being allowed).

It is clear that this defines a Markov process on spatial partitions. Let $q(\bslambda,\bslambda')$ denote the rate of the transition $\bslambda \mapsto \bslambda'$. In order to write it explicitly, observe that it is zero unless there exist $x,y\in\bbZ^{d}$ and $j,k\in\bbN$ (with $k\neq0$) such that
\be   \label{eq:switch} 
	\begin{aligned}
		r_{xj}(\bslambda') & = r_{xj}(\bslambda) - 1,& \quad r_{y,k+1}(\bslambda') & = r_{yk}(\bslambda) - 1,\\
		r_{x,j-1}(\bslambda') & = r_{x,j-1}(\bslambda) +1,&\quad r_{y,k}(\bslambda') & = r_{yk}(\bslambda) - 1 			
	\end{aligned}
\ee
and all other occupation numbers are $r_{zi}$ unchanged. In this case, the rate is
\be
\label{eq:rate}
q(\bslambda,\bslambda') = g(\eta_{x}(\bslambda)) \; t(x,y) \; p_{-}(j,\lambda_{x}) \; p_{+}(k,\lambda_{y}).
\ee
Here, the probability that the customer leaving restaurant $x$ leaves a $j$-table is 
\be \label{eq:pminus} 
	p_-(j;\lambda_x) = \frac{j r_{xj}}{\eta_x}
\ee
and the probability that the customer joins a $k$-table ($k\geq 1$) or opens a new table ($k=0$) is
\be \label{eq:pplus}
	p_+(k;\lambda_y) = \frac{ k r_{yk}}{\eta_y +\theta},\quad p_+(0;\lambda_y) = \frac{\theta}{\eta_x+\theta}. 
\ee

The stationary measure of this process is precisely the measure defined in Eq.\ \eqref{def main prob}. Indeed, it satisfies the detailed balance condition.

\begin{lemma}  \label{lem:detailed-balance} 
	The rate of~Eq.~\eqref{eq:rate} satisfies the detailed balance condition 
	\[
		\bbP_{L,n}(\bslambda) q(\bslambda,\bslambda') = \bbP_{L,n}(\bslambda') q(\bslambda',\bslambda). 
	\]
\end{lemma} 

\begin{proof} 
	We only need to consider the case $\bslambda \neq \bslambda'$. 
	When the customer changes restaurants, we have to check that
	\begin{multline} 
		\bbP_{L,n}(\bslambda) \; g(\eta_x(\bslambda)) \; t(x,y) \; p_-(j;\lambda_x) \; p_+(k;\lambda_y) \\
		= \bbP_{L,n}(\bslambda') \; g(\eta_y(\bslambda')) \; t(y,x) \; p_+(k+1;\lambda'_y) \; p_-(j-1;\lambda'_x). 
	\end{multline} 
Using \eqref{eq:pbalance}, one can check that for all $\bslambda, \bslambda'$ and all $\bseta$, we have the identities
	\be
	\label{detailed balance detailed}
	\begin{split}
		 &\bbP_{L,n}\bigl(\pi_1^{-1}(\{\bseta\})\bigr)  g(\eta_x) t(x,y)
			= \bbP_{L,n}\bigl(\pi_1^{-1}(\{\bseta - \delta_x + \delta_y\})\bigr) g(\eta_y+1) t(y,x), \\
		&\nu_{\eta_x}(\lambda_x) p_-(j;\lambda_x)  = \nu_{\eta_{x} -1}(\lambda'_x) p_+(j-1;\lambda'_x), \\
		&\nu_{\eta_y}(\lambda_y) p_+(k;\lambda_y) = \nu_{\eta_y+1}(\lambda'_y) 
			p_-(k+1;\lambda'_y).
	\end{split}
	\ee
The detailed balance property now follows from Proposition~\ref{prop:conditional}.
\end{proof}

Since the state space $\Lambda_n$ is infinite, we need to check that the Markov process is non-explosive; to this aim we need the following lemma.

\begin{lemma} \label{lem:total-conductance}  
	We have
	\be \label{eq:total-conductance}
		\sumtwo{\bslambda,\bslambda'\in \Lambda_n}{\bslambda \neq \bslambda'} \bbP_{L,n}(\bslambda) q(\bslambda,\bslambda') <\infty. 
	\ee
\end{lemma}

\begin{proof} 
We have
\be
\begin{split}
\sum_{\bslambda'} q(\bslambda,\bslambda') &= \sum_{x,y\in\bbZ^{d}} \sum_{j,k\in\bbN} g(\eta_{x}(\bslambda)) \; t(x,y) \; p_{-}(j,\lambda_{x}) \; p_{+}(k,\lambda_{y}) \\
&= \sum_{x} g(\eta_{x}(\bslambda)) \\
&= \sum_{x} \frac{\eta_{x}}{\theta+\eta_{x}-1} \\
&\leq \frac n\theta.
\end{split}
\ee
\end{proof}

Define $q(\bslambda,\bslambda) =  -\sum_{\bslambda \in \Lambda_n,\, \bslambda' \neq \bslambda} q(\bslambda,\bslambda')$ and  $Q = (q(\bslambda,\bslambda'))_{\bslambda,\bslambda' \in \Lambda_n}$. For later purpose we note 
\be \label{eq:diagonal}
	q(\bslambda,\bslambda) = -  g(\eta_x(\bslambda)) \sum_{y\neq x} t(x,y) - g(\eta_x(\bslambda)) t(x,x) \sum_{k \neq j} p_+(k;\lambda_x^{-j}). 
\ee
Consider the backward Kolmogorov equations 
\be \label{eq:kolmogorov} 
	 \frac{\partial P_t}{\partial t}(\bslambda,\bslambda') = \bigl( Q P_t\bigr) (\bslambda,\bslambda') =  \sum_{\bslambda'' \in \Lambda_n} 
		q(\bslambda,\bslambda'') P_t(\bslambda'',\bslambda'),
\ee
with the standard initial condition $P_0(\bslambda,\bslambda') = \delta_{\bslambda,\bslambda'}$.

\begin{proposition} \label{prop:process}
	There is a unique Markov semi-group $(P_t(\bslambda,\bslambda'))_{t\geq 0}$ solving the backwards Kolmogorov equations for the rates~\eqref{eq:rate}. It defines a Markov process $(\bslambda(t))_{t\geq 0}$ with state space $\Lambda_n$, which has $\bbP_{L,n}$ as a reversible (hence stationary) measure. 
\end{proposition}

Note that $(P_t)$ defines a strongly continuous contraction semi-group in $\ell^\infty(\Lambda_n)$ with infinitesimal generator $Q$: $P_t = \exp(tQ)$.

\begin{proof}
	This follows from Lemmas~\ref{lem:detailed-balance} and~\ref{lem:total-conductance}, 
	and the non-explosion criterion given in~\cite{chen}, Corollary 3.8. 
\end{proof}

\begin{proposition} \label{prop:subprocess} 
	Let $(\bslambda(t))_{t\geq 0}$ be the Markov process from Proposition~\ref{prop:process}, with arbitrary initial law. 
	Then the  marginal $(\bseta(\bslambda(t)))_{t\geq 0}$ is a Markov process 
	with transitions  $(\eta_x,\eta_y) \to (\eta_x - 1,\eta_y+1)$ ($x\neq y$, all other occupation numbers unchanged) happening at rate
	$g(\eta_x) p(x,y)$. 
\end{proposition}

Hence the site occupation numbers evolve according to a (possibly inhomogeneous) zero-range process; the total rate for a customer to leave a restaurant is $g(\eta_x) (1 - p(x,x))$.

\begin{proof} 
	Observe
	\be \label{eq:subprocess1}
		\sum_{\bslambda'\vdash \bseta'} q(\bslambda,\bslambda')  = g(\eta_x(\bslambda)) t(x,y) \sum_{j,k} p_-(j;\lambda_x) p_+(k;\lambda_y) =   g(\eta_x(\bslambda)) t(x,y)
	\ee	
	if $\bseta' = \bseta(\bslambda) - \delta_x + \delta_y$, and in view of Eq.~\eqref{eq:diagonal},
	\be \label{eq:subprocess2}
		\sum_{\bslambda\vdash \bseta} q(\bslambda,\bslambda)  
		= g(x,\eta_x) t(x,x) \sum_{k\neq j} p_+(k;\lambda_x^{-j}) + 
		q(\bslambda,\bslambda) 
		 = -  g(x,\eta_x) \sum_{y \neq x} t(x,y).
	\ee
	Let $\tilde q(\bseta,\bseta'):= g(\eta_x) t(x,y)$ if $\bseta' = \bseta - \delta_x + \delta_y$ for some $x,y\in \bbZ^d$ with $x\neq y$, $\tilde q(\bseta,\bseta') = 0$ if $\bseta' \neq \bseta$ is not of the form just described, $\tilde q(\bseta,\bseta) := 1- \sum_{\bseta' \in \mathcal{N}_n,\bseta' \neq \bseta} \tilde q(\bseta,\bseta')$, and $\tilde Q:= \bigl( \tilde q(\bseta,\bseta') \bigr)_{\bseta,\bseta' \in \mathcal{N}_n} $. An argument similar to the one used in the proof of Proposition~\ref{prop:process} shows that $\tilde P_t:= \exp(t \tilde Q)$ uniquely defines a Markov semi-group on $\mathcal{N}_n$, which is an inhomogeneous zero-range process. Eqs.~\eqref{eq:subprocess1} and~\eqref{eq:subprocess2} become 
	\be
		\sum_{\bslambda' \vdash \bseta'} q(\bslambda,\bslambda') = \tilde q(\bseta(\bslambda),\bseta')  
	\ee
	for all $\bslambda\in \Lambda_n$, $\bseta' \in \mathcal{N}_n$. The latter identity can be conveniently recast in operator language: set  
$
 R(\bslambda,\bseta):= \mathbf{1}_{\{\bslambda \vdash \bseta\}} = \delta_{\bseta,\bseta(\bslambda)}. 
$
We have $Q R = R\tilde Q$, hence $P_t R = R \tilde P_t$ for all $t\geq 0$, i.e., 
\be 
      \sum_{\bslambda' \vdash \bseta'} P_t(\bslambda,\bslambda') = \tilde P_t(\bseta(\bslambda),\bseta').  
\ee	
This condition is sufficient to ensure that $\bseta(\bslambda(t))_{t\geq 0}$ is Markovian with transition rates $\tilde q(\bslambda,\bslambda')$, see Theorem 4 in~\cite{burke-rosenblatt} (the theorem is formulated for finite state spaces but holds in our set-up as well). 
\end{proof}

Thus we have constructed a process combining the zero-range and Chinese restaurant process  for which the measures $\bbP_{L,n}$ are reversible. Finally, let us address the non-Ewens case where the parameters $(\theta_{j})$ depend on $j$. There exist Markov processes such that the measure 
\eqref{def main prob} is stationary. For instance, take a zero-range process on the $(\eta_{x})$, together with an ``instant reshuffling'' of the partitions at the two locations where a customer has been lost or gained. The transition rate of the process is
\[
	g(\eta_x) \; t(x,y) \; \nu_{\eta_{x}-1}(\lambda'_x) \; \nu_{\eta_y+1} (\lambda'_y). 
\]
One can check that Propositions~\ref{prop:process} and \ref{prop:subprocess} remain true. It would be nice to discover other, more natural Markov processes.

\subsection{Coagulation-fragmentation processes}  \label{sec:coag-frag}
For our second process, let $(a_j)_{j\geq 1}$ and $(b_j)_{j\geq 2}$ be positive sequences such that for all $j$, 
\be \label{eq:ajbj}
	a_j \frac{\theta_j}{j} \theta_1 = b_{j+1} \frac{\theta_{j+1}}{j+1}.
\ee
In addition to \eqref{eq:pbalance}, we assume that for some $z>0$ satisfying Eq.~\eqref{eq:zfinite}, we have
\be
	\sum_{x,y\in\bbZ^d} \e{-V(x/L)} t(x,y)<\infty,\quad 
		\sum_{x\in \bbZ^d,j\geq 2} b_{j} \frac{\theta_{j}}{j} z^j \e{- j V(x/L)} <\infty.  
\ee
As the process presented here is a generalization of the stochastic version of the Becker--D{\"o}ring model presented in Sec.~\ref{sec:becker-doering}, we use the vocabulary of nucleation/particle clustering and we speak of particles and clusters (or droplets) instead of customers or restaurants. 
 \emph{Monomers} are clusters of size $1$.
We allow three types of transitions: 
\begin{itemize}
	\item \emph{Coagulation} at a given site $x$: a $j$-cluster ($j\geq 1$) and a $1$-cluster coagulate to a $(j+1)$-cluster, i.e., $(r_{x,j-1},r_{j,x},r_{x,j+1})\to (r_{x,j-1}-1,r_{x,j}-1,r_{x,j+1}+1)$. 
	This occurs at rate $a_j r_{xj}(\bslambda) r_{x1}(\bslambda)$ if $j \neq 1$, and at rate $a_{1} r_{x1}(\bslambda) (r_{x1}(\bslambda) - 1)$ if $j =1$. 
	\item \emph{Fragmentation} at a given site $x$: a monomer departs from a $j$-cluster ($j\geq 2$). 
	The transition $(r_{x1},r_{x,j-1},r_{x,j}) \to (r_{x1}+1,r_{x,j-1}+1,r_{x,j}-1)$ 
	occurs at rate $b_{j} r_{xj}(\bslambda)$. 
	\item \emph{Jump of a monomer} from site $x$ to site $y\neq x$:  the transition 
	$(r_{x1},r_{y1}) \to (r_{x1}-1,r_{y1}+1)$ occurs at rate $r_{x1}(\bslambda) t(x,y)$.  
\end{itemize} 
There are many possible generalizations: for example, we could allow for the jumping of whole clusters rather than only monomers, or allow two groups of size larger than one to coalesce, or take coagulation-fragmentation rates proportional to powers of the occupation numbers $r_{xj}^\gamma$, see \cite{durrett-granovsky-gueron}. We stick to the process presented here for notational simplicity, and also because it is closest in spirit to the original Becker--D{\"o}ring model of nucleation~\cite{becker-doering} set in the framework of kinetic gas theory: monomers are mobile particles of a dilute gas, clusters of size $j\geq 2$ are small chunks of condensed precipitate and they are deemed immobile (or at least very slow compared to gas particles). 

\begin{proposition} \label{prop:coag-frag-process} 
	The transition rates specified above define uniquely a Markov semi-group with state space $\Lambda_n$. The measure $\bbP_{L,n}$ is reversible, hence stationary, for the associated Markov process. 	
\end{proposition}

\begin{proof}
We leave the proof of the detailed balance conditions to the reader and look at the non-explosion criterion from Lemma~\ref{lem:total-conductance}.
Let $z>0$ be such that Eq.~\eqref{eq:zfinite} holds. The contribution to $\sum_{\bslambda\neq \bslambda'} \bbP_{L}^z(\bslambda) q(\bslambda,\bslambda')$ from coagulation of $j$-clusters ($j\neq 1$)
with monomers is estimated as 
\begin{multline} 
	\sum_{\bslambda\in \cup_n \Lambda_n} \bbP_L^z(\bslambda) \sum_{x\in \bbZ^d, j\in \bbN} a_j r_{xj}(\bslambda) r_{x1}(\bslambda) 
		 = \sum_{x\in \bbZ^d, j\in \bbN} a_j \Bigl(\theta_1 z\e{-V(x/L)}\Bigr) 
			\Bigl( \frac{\theta_j}{j} z^j \e{- j V(x/L)} \Bigr)\\
	= \sum_{x\in \bbZ^d,j\geq 1} b_{j+1}  \frac{\theta_{j+1}}{j+1} z^{j+1} \e{- (j+1) Vx/L)} <\infty.
\end{multline}
	We have used that under $\bbP_{L}^z$ the $R_{xj}$s are independent Poisson variables, and we have used the assumption~\eqref{eq:ajbj}.
The contributions of coagulation of monomers, fragmentation, and random walk can be evaluated in a similar way. 
 We conclude as in Lemma~\ref{lem:total-conductance} and Proposition~\ref{prop:process}. 
\end{proof}

It is natural to ask about the evolution of the marginals $\bseta(\bslambda(t))$ and $\bsr(\bslambda(t))$. Let us first look at the site occupation numbers. The rate for a particle to leave a site depends on the number of monomers at this site, i.e., it is not a function of the total number of particles at the site alone. Therefore $\bseta(\bslambda(t))$ is not Markovian. However, it is natural to approximate the process by replacing the jump rate $r_{x1}t(x,y)$ by the conditional expectation 
\be \label{eq:eff-zrp}
	t(x,y) \bbE_{L,n}\bigl( r_{x1}(\bslambda) \big| \eta_x(\bslambda) = \eta_x\bigr) 
		= t(x,y) \sum_{r_{x1}\geq 1} r_{x1}  \nu_{\eta_x}( r_{x1}) = \frac{\theta_1 h_{\eta_x-1}}{h_{\eta_x}} t(x,y),
\ee 
which is  the transition rate for  a zero-range process (for the last equality, see Proposition~2.1 in~\cite{ercolani-ueltschi}).  
Now look at the cluster size counts $r_{j} = \sum_x r_{xj}$. Remember that the conditional law of $r_{xj}$ is binomial with $r_j$ trials and success probability $p_{xj} \propto \exp( - jV(x/L))$. 
Again, $\bsr(\bslambda(t))_{t\geq 0}$ will not be Markovian, but it is natural to approximate it by a process with effective coagulation rate 
\be \label{eq:eff-coagj}
	\sum_{x\in \bbZ^d} a_j \bbE_{L,n} \Bigl[ r_{x1}(\bslambda) r_{xj}(\bslambda) \Big| r_1(\bslambda) = r_1,\ r_j(\bslambda) = r_j \Bigr] 
	 = a_j r_1 r_j \sum_{x\in \bbZ^d} p_{x1} p_{xj},  
\ee
for $j\neq 1$. Similar expressions hold for coagulation of two monomers and for the effective fragmentation rate. 
In the case of square traps, $\exp( - V(x/L))=\mathbf{1}_{[0,L)^d}(x)$ with $L\in \bbN$, we have  $p_{xj} = L^{-d} \mathbf{1}_{[0,L)^d}(x)$ and the effective rates are exactly those of the Becker--D{\"o}ring process from Section~\ref{sec:becker-doering}.

\section{Condensation}  \label{sec rv}  

We study the asymptotic behavior in the thermodynamic limit $n\to \infty$, $L\to \infty$ at fixed density $\rho = n/L^d$. The model may undergo a {\it phase transition} when the density varies. Above a certain {\it transition density} $\rho_{\rm c}$, a {\it condensation} occurs that is characterized by the presence of large components.

There are two natural definitions of condensation, one in terms of site occupation numbers (number of customers of a restaurant), used in the zero-range process, and the other in terms of random partitions (number of occupants of a table). We shall actually observe that both definitions are equivalent. The results of this section can be found in Theorems \ref{thm trans dens} and \ref{thm infinite elements}. The expressions are familiar for the zero-range process and for spatial permutations. The motivation of this section is to present a novel, unified perspective.

Condensation is measured by the following ``order parameters'':
\begin{align}
\label{def nu oo}
	\nu_\infty &=\lim_{K\to\infty} \limtwo{n\to\infty}{\rho L^{d}=n} \bbE_{L,n} \Bigl( \frac1n \sum_{j>K} j \, R_{j} \Bigr) \\
	\mu_\infty&= \lim_{K\to\infty} \limtwo{n\to\infty}{\rho L^{d}=n} \bbE_{L,n}\Bigl( \frac1n \sum_{x : \eta_{x}>K} \eta_{x}(\bslambda) \Bigr).
\end{align} 
(The existence of the limits over $n$ is proved in the proof of Theorem~\ref{thm infinite elements} below. The existence of the limits over $K$ is guaranteed by monotonicity.)
 We say that condensation occurs if $\mu_\infty>0$ or $\nu_\infty>0$, and refer to $\mu_\infty$ or $\nu_\infty$ as the \emph{condensate fraction}. 
In principle, it might happen that $ \nu_\infty< \mu_\infty$. But Theorem~\ref{thm trans dens} below shows that $\nu_\infty=\mu_\infty$; the two definitions of condensation are equivalent. 

Under some assumptions, we prove the existence of a transition density $\rho_{\rm c}$ such that the order parameters are zero for $\rho \leq \rho_{\rm c}$ and positive for $\rho > \rho_{\rm c}$. In order to define the transition density, let
\be \label{eq:rhoz}
	\rho(z) = \sum_{j\geq 1} \theta_j z^j \int_{\bbR^d} \e{- j V(x)} \dd x.  
\ee
Let $z_{\rm c} \in (0,\infty]$ be the radius of convergence of the power series above. The \emph{transition density} is
 \be
 \label{def crit dens}
 \rho_{\rm c} = \rho(z_{\rm c}).
 \ee
 It is possible that $\rho_{\rm c}$ is infinite, meaning that no transition takes place. In the case $z_{\rm c}=1$ we get
 \be
 \label{Einstein}
 \rho_{\rm c} = \sum_{j\geq1} \theta_{j} \int_{\bbR^{d}} \e{-jV(x)} \dd x.
 \ee
 The transition density may be finite for two reasons (or a combination of both). First, a trap such as $V(x) = \|x\|^{\delta}$ yields $\int \e{-jV} \sim j^{-d/\delta}$, which is summable if $\delta<d$. This is the case of the ideal Bose gas of Section \ref{sec ideal Bose}, where $\theta_{j} \equiv 1$ and $\delta=2$. Eq.~\eqref{Einstein} is then the famous formula derived by Einstein in 1924. Second, the parameters $\theta_{j}$ may be summable. This is the case of the  invariant measure of the zero-range process of Section~\ref{sec zero-range} in the regime studied by Evans \cite{evans96}. There is no confining potential here, $\int \e{-jV} \equiv 1$. A very different situation is particle clustering, see Section \ref{sec:droplets}, where the parameters $\theta_{j}$ are given by certain cluster integrals \cite{jkm}. The general form of \eqref{eq:rhoz} with $\theta_{j} \not\equiv 1$ and $V \not\equiv 0$ appeared in the context of spatial random permutations with cycle weights \cite{BU-prl,BU2} discussed in Section \ref{sec spatial permutations}. A variant of the alternative expression~\eqref{eq:rhoz} for zero-range processes with disorder can be found e.g.\  in~\cite{godreche-luck} (Eq.\ (4.2)).

If we are given parameters $h_{n}$ instead of the $\theta_{j}$, we can consider the alternative expression
\be \label{eq:rhoz-hk}
	\rho(z)= \int_{\bbR^d} \frac{\sum_{n\geq 1} n h_n z^n \e{-n V(x)} }{\sum_{n\geq 0} h_n z^n \e{-n V(x)}}\, \dd x,
\ee
where the integrand is by definition equal to $0$ when $V(x) = \infty$. Indeed, the right side of Eq.~\eqref{eq:rhoz-hk} is equal to
\[
\int_{\bbR^d} z \frac{\partial}{\partial z} \log \Bigl( \sum_{k\geq 0} h_k z^k \e{- k V( x)} \Bigr) \dd x,
\]
and using \eqref{eq:levy} with $z \mapsto z \e{-V(x)}$ we get $\rho(z)$ in \eqref{eq:rhoz}.
The expression \eqref{eq:rhoz-hk} is the density-activity series for the (inhomogeneous) zero-range processes. 

We make the following assumptions throughout this section.

\begin{assumption}
\label{ass 1}
Assume that $V : \bbR^{d} \to [0,\infty]$ is continuous at $0$, with $V(0)=0$, and that for every $j\geq1$, we have
\[
\frac1{L^{d}} \sum_{x\in\bbZ^{d}} \e{-j V(x/L)} \to \int_{\bbR^{d}} \e{-j V(x)} \dd x
\]
as $L\to\infty$.
\end{assumption}

It follows from this assumption that the series $\sum_{j\geq1} \frac{\theta_{j}}j z^{j}$ has the same radius of convergence $z_{\rm c}$ as the series \eqref{eq:rhoz} for $\rho(z)$.
The next theorem claims that $\rho_{\rm c}$ is indeed the transition density.

\begin{theorem}[Condensation]
\label{thm trans dens}
Under Assumption \ref{ass 1}, we have  
	\[ \label{eq:equivalence}
		\nu_\infty = \mu_\infty =  \max\bigl(0, 1 - \tfrac{\rho_{\rm c}}\rho\bigr). 
	\]
\end{theorem}

The proof of this theorem can be found immediately after Theorem \ref{thm infinite elements} below.

It is possible to provide more detailed results about ``typical sizes'' of components. Choose a customer at random, and consider the size of the element of the partition that he belongs to. Informally, we have
\be
\text{Prob(typical size of element is $j$)} = \bbE_{L,n}\Bigl( \frac{j R_{j}}n \Bigr).
\ee
Indeed, there are $j R_{j}$ customers in tables with $j$ people. Similarly, let $N_{k}$ denote the number of restaurants with $k$ people:
\be
N_k(\omega) = \#\{ x\in \bbZ^d:\  H_x(\omega) = k\}.
\ee
The probability that a random customer finds himself in a restaurant with $k$ people is
\be
\text{Prob(typical site occupation is $k$)} = \bbE_{L,n}\Bigl( \frac{N_{k}}n \Bigr).
\ee

The next theorem gives the asymptotic behavior of typical sizes. In order to state it, we introduce $z_{0}(\rho)$ as the unique solution of the equation $\rho(z) = \rho$ when $\rho < \rho_{\rm c}$; we set $z_{0}(\rho) = z_{\rm c}$ when $\rho \geq \rho_{\rm c}$.

\begin{theorem}[Typical sizes]
\label{thm infinite elements}
Under Assumption \ref{ass 1}, we have for all $j\geq1$ and $k\geq0$,
\begin{itemize}
\item[(a)] $\displaystyle \frac{j R_j}{n} \stackrel{\mathrm p}{\longrightarrow} \frac{1}{\rho} \theta_j z_0(\rho) ^j \int_{\bbR^d} \e{- j V(x)} \dd x$,
\item[(b)] $\displaystyle \frac{N_k}{n} \stackrel{\mathrm p}{\longrightarrow} \frac{1}{\rho} \int_{\bbR^d} \frac{k h_k z_0(\rho)^k \exp( - k V(x))}{\sum_{j\geq 0}  h_j z_0(\rho)^j \exp( -jV(x))}  \dd x$. \end{itemize}
Here, $\stackrel{\mathrm p}{\longrightarrow}$ denotes the convergence in probability as $n,L\to\infty$ with $\rho L^{d}=n$.
\end{theorem}

The proof of this theorem can be found at the end of the section. We first use it to prove Theorem \ref{thm trans dens}.

\begin{proof}[Proof of Theorem \ref{thm trans dens}]
We have
\be
\begin{split}
\nu_{\infty} &= 1 - \lim_{K\to\infty} \lim_{L,n\to\infty} \sum_{j=1}^{K} \bbE_{L,n} \Bigl( \frac{j R_{j}}n \Bigr) \\
&= 1 - \frac1\rho \sum_{j\geq 1} \theta_j z_0(\rho)^j \int_{\bbR^d} \e{-j V(x)} \dd x
\end{split}
\ee
and the definitions of $\rho_{\rm c}$ and $z_0(\rho)$ imply that $\nu_\infty= \max (0, 1- \frac{\rho_{\rm c}}{\rho})$. The proof for $\mu_\infty$ is similar and is based on the alternative formula~\eqref{eq:rhoz-hk} for $\rho(z)$. 
\end{proof}

The rest of this section is devoted to the proof of Theorem \ref{thm infinite elements}. The method relies on the conditioning relation~\eqref{eq:pln-cond}; we first prove limit laws for the grand-canonical measures $\bbP_L^z$ and then show that the conditioning merely picks the right activity $z= z_0(\rho)$. Technicalities arise because the grand-canonical probability to have exactly $n$ particles goes to $0$, and because the power series at finite $L$ need not converge at $z= z_{\rm c}$. 

Lemma~\ref{lem:zln} below and the limit law for the $R_j$'s are closely related to large deviations results~\cite{BLP,Suto2,BCMP,GSS,jkm}; we provide a complete proof because our setting is more general.

\begin{lemma} 
\label{lem:zln}
\[
\limtwo{n\to\infty}{\rho L^{d}=n} \frac{1}{n} \log Z_{L,n} 
= \frac{1}{\rho} \sum_{j\geq 1}\frac{\theta_j}{j} z_0(\rho)^j \int_{\bbR^d} \e{-j V(x)} \dd x  -  \log z_0(\rho). 
\]		
\end{lemma}

\begin{proof}
Using Eq.~\eqref{eq:zln-cond} but neglecting the probability term, we get
\be
\frac{1}{L^d} \log Z_{L,n} \leq \sum_{j\geq 1} \frac{\theta_j}{j} z^j \frac{1}{L^d} \sum_{x\in \bbZ^d} \e{- j V(x/L)} - \rho \log z. 
\ee
For any $z<z_{\rm c}$, we have by dominated convergence
\be
\limsuptwo{n\to \infty}{ \rho L^d = n} \frac{1}{L^d} \log Z_{L,n} \leq 	\sum_{j\geq 1} \frac{\theta_j}{j} z^j \int_{\bbR} \e{- j V(x)} \dd x - \rho \log z. 
\ee
The bound extends to $z=z_{\rm c}$ by continuity. The equation of Lemma \ref{lem:zln} holds at least as an upper bound.

For the lower bound, we start with Proposition~\ref{prop:marginals} (b). We have
\be 
Z_{L,n} \geq \prod_{j=1}^n \frac{1}{r_j!} \Bigl( \frac{\theta_j}{j} \frac{1}{L^d} \sum_{x\in \bbZ^d} \e{-j V(x/L)} \Bigr)^{r_j}
\ee
for any choice of non-negative integers $(r_{j})$ such that $\sum j r_{j} = n$. Let us introduce the following entropy function on sequences $\bsa = (a_{j})$ such that $a_{j} \geq 0$ and $\sum_{j} j a_{j} \leq 1$:
\be
I(\bsa) = -\sum_{j\geq1} a_{j} \log \Bigl( \frac{\e{} \theta_{j}}{j a_{j} \rho} \int_{\bbR^{d}} \e{-j V(x)} \dd x \Bigr). 
\ee
We have 
\be
\label{borne inferieure}
\frac{1}{n} \log Z_{L,n} \geq - I\bigl(\tfrac1n r_{1}, \tfrac1{n} r_{2}, \dots \bigr) + \sum_{j\geq1} \frac{r_{j}}n \log \frac{\frac1{L^{d}} \sum_{x \in \bbZ^{d}} \e{-j V(x/L)}}{\int_{\bbR^{d}} \e{-j V(x)} \dd x} - \frac{1}{n} \sum_{j=1}^{n} \log \frac{r_j!}{(r_j/{\rm e})^{r_j}}.
\ee
Let $a_j$ be equal to the right side of Theorem \ref{thm infinite elements} (a). We choose a sequence $\bsr^{(n)}$ such that $\sum j r_{j}^{(n)} = n$, $\frac1n r_{j}^{(n)} \leq 2a_{j}$, and $\frac1n r_{j}^{(n)} \to a_{j}$ as $n\to\infty$. By continuity of $V$ around $0$, we have
\be
\begin{split}
& {\rm const} \; j^{-d} \leq \int \e{-j V(x)} \dd x \leq \int \e{-V(x)} \dd x, \\
& {\rm const} \; j^{-d} \leq \frac1{L^{d}} \sum_{x} \e{-j V(x/L)} \leq \frac1{L^{d}} \sum_{x} \e{-V(x/L)} \leq {\rm const},
\end{split}
\ee
where the constants do not depend on $L,j$. The logarithm in the first sum in Eq.~\eqref{borne inferieure} is less than a constant times $j$, so the summand is less than a constant times $j a_{j}$. We can then use the dominated convergence theorem and the first sum in Eq.\ \eqref{borne inferieure} vanishes in the limit $n\to\infty$. The last sum also vanishes: Using $1 \leq \frac{r_{j}!}{(r_{j}/\e{})^{r_{j}}} \leq C r_{j}$, we have for any $K$
\be
\frac1n \sum_{j=1}^{K} \log \frac{r_{j}!}{(r_{j}/\e{})^{r_{j}}} \leq \frac1n \sum_{j=1}^{K} \log Cn,
\ee
which goes to 0 as $n\to\infty$. Further, using $\log s \leq s$,
\be
\frac1n \sum_{j=K+1}^{n} \log \frac{r_{j}!}{(r_{j}/\e{})^{r_{j}}} \leq \frac1n \sum_{j=K+1}^{n} \log C r_{j} \leq \frac CK \sum_{j=K+1}^{n} \frac{j r_{j}}n \leq \frac CK.
\ee
This is arbitrarily small by choosing $K$ large enough. Finally, recall that $I$ is continuous on the set of non-negative sequences $(b_{j})$ satisfying $\sum j b_{j} \leq 1$ \cite{ball-carr-penrose}. Then
\be
\liminf_{n\to\infty} \frac1n  \log Z_{L,n} \geq -I(\bsa),
\ee
and a calculation shows that $-I(\bsa)$ is equal to the right side of the expression in Lemma \ref{lem:zln}. 
\end{proof}

Recall the definition of $\bbE_{L}^{z}$ given in Section \ref{sec:poisson}.

\begin{lemma}
\label{lem z-convergence}
For any $z<z_{\rm c}$, we have
\begin{itemize}
\item[(a)] $\displaystyle \lim_{L\to\infty} \bbE_{L}^{z} \Bigl( \frac{jR_{j}}{L^{d}} \Bigr) = \theta_{j} z^{j} \int_{\bbR^{d}} \e{-j V(x)} \dd x$.
\item[(b)] $\displaystyle \lim_{L\to\infty} \bbE_{L}^{z} \Bigl( \frac{N_{k}}{L^{d}} \Bigr) = \int_{\bbR^{d}} \frac{h_{k} z^{k} \e{-kV(x)}}{\sum_{j\geq1} h_{j} z^{j} \e{-jV(x)}} \dd x$.
\end{itemize}
\end{lemma}

\begin{proof}
We have
\be
\bbE_{L}^{z} \Bigl( \frac{jR_{j}}{L^{d}} \Bigr) = \theta_{j} z^{j} \frac1{L^{d}} \sum_{x\in\bbZ^{d}} \e{-jV(x)},
\ee
and the limit (a) exists by Assumption \ref{ass 1}. For the limit (b), we have
\be
\begin{split}
\bbE_{L}^{z} \Bigl( \frac{N_{k}}{L^{d}} \Bigr) &= \frac1{L^{d}} \sum_{x \in \bbZ^{d}} \frac{h_{k} z^{k} \e{-kV(x/L)}}{\sum_{j} h_{j} z^{j} \e{-jV(x/L)}} \\
&= \frac1{L^{d}} \sum_{x \in \bbZ^{d}} h_{k} z^{k} \e{-kV(x/L)} \exp\Bigl\{ -\sum_{j} \frac{\theta_{j}}j z^{j} \e{-jV(x/L)} \Bigr\}.
\end{split}
\ee
We can expand the last exponential so as to get a convergent series of terms of the form $\e{-nV(x/L)}$ with $n\in\bbN$. The Riemann sums converge in the limit $L\to\infty$ by Assumption 1, and the result follows from dominated convergence.
\end{proof}

\begin{lemma}
\label{lem:deviations-bound}
Let $z<z_{\rm c}$ and $a>0$. There exist $L_{0}$ and $b>0$ such that for all $L\geq L_{0}$, we have
\begin{itemize}
\item[(a)] $\displaystyle \bbP_L^z\Bigl( \tfrac{1}{L^d} \bigl| N_k - \bbE_L^z N_k\bigr|  > a \Bigr) \leq \e{-b L^d}$,
\item[(b)] $\displaystyle \bbP_L^z\Bigl( \tfrac{1}{L^d} \bigl| R_j - \bbE_L^z R_j \bigr|  > a \Bigr) \leq \e{-b L^d}$.
\end{itemize}
\end{lemma}

\begin{proof}
We bound the cumulant generating functions and then deduce bounds with the help of Markov's inequality. Since random variables at different locations are independent, we have
\be
\begin{split}
\frac{1}{L^d}  \log \bbE_L^{z} \bigl[  \e{t ( N_k - \bbE_L^z N_k)} \bigr] 
	 &= \frac{1}{L^d}  \sum_{x\in \bbZ^d} \log  \bbE_L^{z} \Bigl[ \exp\Bigl( t \mathbf{1}_{\{H_x(\omega) =k \}} - t \; \bbP_L^z(H_x=k) \Bigr) \Bigr] \\
	&= \frac{1}{L^d} \sum_{x\in \bbZ^d} \Bigl\{ -t \bbP_L^z(H_x=k) + \log \Bigl(1 +  (\e{t}-1) \; \bbP_L^z(H_x=k) \Bigr) \Bigr\}.
\end{split}
\ee
Using $\log(1+s) \leq s$, we find
\be
\begin{split}
\frac{1}{L^d}  \log \bbE_L^{z} \bigl[  \e{t ( N_k - \bbE_L^z N_k)} \bigr] &\leq \bigl( \e{t} - 1 - t \bigr) \frac{1}{L^d} \sum_{x\in \bbZ^d} \bbP_L^z(H_x=k) \\
&= \bigl( \e{t} - 1 - t \bigr) \; \bbE_L^z \Bigl( \frac{N_{k}}{L^{d}} \Bigr).
\end{split}
\ee
The latter expectation converges by Lemma \ref{lem z-convergence} (b) so that everything is bounded by $c(z) t^{2}$ for $t\leq1$ with $c(z)$ depending on $z$ only. Then
\be
\begin{split}
\bbP_{L}^{z} \Bigl( \tfrac1{L^{d}} (N_{k} - \bbE_{L}^{z} N_{k} ) > a \Bigr) &= \bbP_{L}^{z} \Bigl( \e{t (N_{k} - \bbE_{L}^{z} N_{k})} > \e{ta L^{d}} \Bigr) \\
&\leq \e{-ta L^{d}} \bbE_{L}^{z} \Bigl( \e{t (N_{k} - \bbE_{L}^{z} N_{k})} \Bigr) \\
&\leq \e{-L^{d} (ta - c(z) t^{2})}.
\end{split}
\ee
This is smaller than $\frac12 \e{-b L^{d}}$ for some $b>0$. A similar bound can be found for $\bbP_{L}^{z}(\tfrac1{L^{d}} (N_{k} - \bbE_{L}^{z} N_{k} ) < -a)$ and this completes the proof of (a). The same method can be used to prove the claim (b) for $R_{j}$.
 \end{proof}

\begin{proof}[Proof of Theorem \ref{thm infinite elements}]
We only prove the claim (b), as the proof for (a) is the same. Let $m_{k}$ be the limit in (b). We have 
\be
\begin{aligned} 
	\bbP_{L,n}\Bigl( \bigl|\tfrac{N_k}{L^d} - m_k \bigr| > \eps\Bigr) 
	&\leq  \bbP_{L,n}\Bigl( \bigl|\tfrac{N_k}{L^d} - \bbE_{L}^{z}(\tfrac{N_{k}}{L^{d}}) \bigr| > \eps - \bigl| m_k - \bbE_{L}^{z}(\tfrac{N_{k}}{L^{d}}) \bigr| \Bigr) \\
	&= \frac{\bbP_{L}^z\bigl( \bigl|\frac{N_k}{L^d} -\bbE_{L}^{z}(\tfrac{N_{k}}{L^{d}}) \bigr| > \eps - \bigl| m_k - \bbE_{L}^{z}(\tfrac{N_{k}}{L^{d}}) \bigr| \bigr)} {\bbP_L^z\bigl(\sum_{j\geq 1} j R_{j} = n\bigr)}.
\end{aligned}
\ee
Then
\be
\bbP_{L}^{z} \Bigl( \sum_{j\geq1} j R_{j} = n \Bigr) = z^{n} Z_{L,n} \exp\Bigl\{-\sum_{x\in\bbZ^{d}} \sum_{j\geq1} \frac{\theta_{j}}j z^{j} \e{-jV(x/L)} \Bigr\}.
\ee
We choose $z$ close to $z_{0}(\rho)$ in such a way that
\be
\label{eq:z-choice}
\begin{split}
&\Bigl| m_{k} - \int_{\bbR^{d}} \frac{h_{k} z^{k} \e{-kV(x)}}{\sum_{j\geq1} h_{j} z^{j} \e{-jV(x)}} \dd x \Bigr| \leq \frac\eps2, \\
&\Bigl| \sum_{j\geq 1} \frac{\theta_j}{j}\bigl( z^j - z_0(\rho)^j)\int_{\bbR^d} \e{-j V(x)}\dd x - 
\rho \log \frac{z}{z_0(\rho)} \Bigr| \leq \frac a2.
\end{split}
\ee
(If $\rho < \rho_{\rm c}$, we can choose $z = z_0(\rho)$; if $\rho\geq  \rho_{\rm c}$, we need to choose $z$ close enough to $z_{\rm c}$.)
From Lemma~\ref{lem:deviations-bound} there exist $b>0$ and $L_{0}$ such that for all $L>L_{0}$, we have
\be
\bbP_{L}^z\Bigl( \bigl|\tfrac{N_k}{L^d} -\bbE_{L}^{z}(\tfrac{N_{k}}{L^{d}}) \bigr| > \eps - \bigl| m_k - \bbE_{L}^{z}(\tfrac{N_{k}}{L^{d}}) \bigr| \Bigr) \leq \e{-bL^{d}}.
\ee
It follows from Lemma \ref{lem:zln} that there exists $L_{1}$ such that if $L>L_{1}$, we have
\be
\bbP_{L}^{z} \Bigl( \sum_{j\geq1} j R_{j} = n \Bigr) \geq \e{-a L^{d}}.
\ee
Choosing $a<b$, we get the claim.
\end{proof}

\section{Trap potentials}\label{sec:traps} 

In this section we investigate the shape and location of the condensate for  a class of potential functions that have a unique minimum at the origin; square traps are treated in Section~\ref{sec:square}. Theorem~\ref{thm H0} states that the condensate is located at the origin (the trap's minimum), and Theorems~\ref{thm clt}--\ref{thm:alg-neg-fluct} provide the fluctuations of the occupation of the origin. The fluctuations are governed by infinitely divisible laws that need not be normal or $\alpha$-stable, as expected for sums of independent random variables that are not identically distributed~\cite{gnedenko-kolmogorov}. Theorems~\ref{thm H0}, \ref{thm clt} and \ref{thm:ewens-fluct} generalize known results~\cite{BU2,BP,CD}, Theorems~\ref{thm:alg-fluct} and~\ref{thm:alg-neg-fluct} are new. A graphical overview is given in Figure~\ref{fig:fluct}. Section~\ref{sec:phase-diagram} discusses the distribution of the partition elements at the origin: the condensate can concentrate on a single large element or be distributed according to some non-trivial law,  e.g.\ a Poisson-Dirichlet distribution.

\subsection{Macroscopic occupation of the origin}
In this subsection we consider the marginal measure on $\bseta$ and we establish that the excess mass is concentrated at $x=0$. We also prove a central limit theorem.

\begin{assumption}
\label{ass 2}
We assume that the potential function $V: \bbR^{d} \to [0,\infty]$ satisfies the following properties:
\begin{itemize}
\item[(i)] $V(x) = \|x\|^{\delta} (1+o(1))$ around the origin with $0<\delta<d$.
\item[(ii)] $V(x)>b>0$ for all $\|x\|>1$.
\item[(iii)] For every $a>0$, there exists a constant $C_{a}$ such that
\[
\frac1{L^{d}} \sum_{x \in \bbZ^{d} \setminus \{0\}} \e{-a V(x/L)} < C_a \int_{\bbR^{d}} \e{-aV(x)} \dd x < \infty;
\]
we also assume that the left side converges to $\int \e{-aV}$ as $L\to\infty$.
\end{itemize}
For the weights, we assume that $\theta_j>0$ for all $j$, that $\lim_{j\to \infty} \frac{1}{j}\log \theta_j = 0$, and that
\be \label{eq:traptheta-cond1}
	\sum_{j\geq 1} \frac{\theta_j}{j^{d/\delta}} <\infty.  
\ee
\end{assumption}

Notice that if $V(x) \approx c\|x\|^{\delta}$ with $c>0$ around the origin, there is no loss in generality in taking $c=1$. Indeed, let $L' = c^{-1/\delta} L$ and $V'(x) = V(c^{-1/\delta} x)$. Then $V'(x) \approx \|x\|^{\delta}$ around 0 and the probability can be written in terms of $L'$ and $V'$ by replacing $\e{-j V(x/L)}$ by $\e{-j V'(x/L')}$.

It follows from our assumptions that  $\int_{\bbR^d} \exp( -  j V(x)) \dd x$ is bounded from above and below by a constant times $j^{-d/\delta}$, whence  $z_{\rm c}=1$ and $\rho_{\rm c}<\infty$. Additional regularity conditions on the $\theta_j$s, formulated as conditions on the tails of the $h_n$s, will be imposed in the theorems. They are loosely related to regularity conditions for  heavy-tailed random variables~\cite{embrechts-klueppelberg-mikosch} and can be checked, in part, with the help of Theorem~\ref{thm buv}. 

The assumption $\delta<d$ will be important in our proofs. We suspect that a condition on $\delta$ is necessary, as for large $\delta$ the potential may be too shallow in order to confine the condensate to a single site; in fact the limit $\delta\to \infty$ corresponds, formally, to the square traps from Section~\ref{sec:square}, where all sites are alike and the condensate chooses uniformly among them. It is not clear how large $\delta$ should be in order for this to occur.

Recall that $H_{0}$ is the random variable for the total occupation of the origin.

\begin{theorem}
\label{thm H0}
Suppose that Assumption \ref{ass 2} holds true, and that there exist constants $C, c \geq 0$ and $a \leq \frac12 (1-\frac\delta d)$ such that for all $n\geq1$,
\[
C^{-1} \e{-c  n^a} \leq h_n \leq  C \e{c n^a}.
\]
Assume also that $\rho>\rho_{\rm c}$. Then as $n,L\to\infty$ with fixed $\rho = n/L^{d}$, we  have the convergence in distribution
\[
	\frac1{L^{d}} H_{0} \todi \rho - \rho_{\rm c}.
\]
\end{theorem}

The theorem applies to algebraically decaying  weights $\theta_j = j ^{-\gamma}$, $\gamma>0$, by Theorem~\ref{thm buv}.  It also applies to stretched exponential weights $\theta_j /j = \e{- j ^\gamma}$  and algebraically growing weights $\theta_j = j^\gamma$ with $\gamma>0$ small enough. In the latter case, we indeed have
\be \label{eq:hn-alg-growth}
	h_n  = \bigl(1+o(1)\bigr) c n^{- (\gamma+2)/[ 2(\gamma+1)]}  \exp\Bigl( C n^{\gamma/ (1+\gamma)} \Bigr) 
\ee 
for suitable constants $c,C$. 
Eq.~\eqref{eq:hn-alg-growth} was proven by Erlihson and Granovsky, see Eq. (4.66) in ~\cite{erlihson-granovsky}. It also follows from results proven independently in~\cite{ercolani-ueltschi}. (The proof in \cite{ercolani-ueltschi} was given for specific weights that satisfy $\theta_j = (1+o(1)) j^\gamma$, but it can be extended with the help of known results on polylogarithms~\cite[Chapter VI.8]{FS}.)

For the proof of Theorem~\ref{thm H0} we follow \cite{BP,BU2}, see also~\cite{CD}.  First we express the canonical expectations with respect to the grand-canonical measure at $z = z_{\rm c} =1$  with the help of the conditioning relation~\eqref{eq:pln-cond}. A difficulty here is that $\sum_j \theta_j /j$ can be infinite, in which case 
$\bbP_{L}^{z_{\rm c}}( H_0 = \infty) =1$. Therefore we give a special treatment to $x=0$. Let  
\be
	M(\omega) = \sum_{x\in \bbZ^d\backslash \{0\}} H_x(\omega) = \sum_{j\geq 1} \sum_{x\in \bbZ^d\backslash \{0\}} j R_{xj}(\omega)
\ee
be the number of particles not at $0$.\footnote{A careful inspection of the definitions shows that the measure $\mu_\Lambda$ used by Buffet and Pul{\'e}~\cite{BP} and Betz and Ueltschi~\cite{BU2} is, in our notation,  the law of $M$ under $\bbP_L^{z_{\rm c}}$. The variable $M$ is also used by Chatterjee and Diaconis~\cite{CD}.}
Let $\rho_{\rm c}^{L} = \frac{1}{L^d} \bbE_{L}^{z_c}(M)$; this is approximately equal to the critical density:
\be
\rho_{\rm c}^L = \frac{1}{L^d} \bbE_{L}^{z_c}(M) = \frac{1}{L^d} \sum_{j\geq 1} \sum_{x\in \bbZ^d\backslash \{0\}} \theta_j  \e{- j V(x/L)}, 
\ee
which converges to $\rho_{\rm c}$ by dominated convergence.

\begin{lemma} \label{lem:m-laplace}
Under the same assumptions as in Theorem \ref{thm H0}, there exist $C,c>0$ such that for all $L>0$ and all $B>0$, 
\[
\bbP_L^{z_{\rm c}}\Bigl( \bigl| \tfrac1{L^{d}} M -\rho_{\rm c}^{L} \bigr| \geq B \Bigr) \leq C \exp( - c \, B \, L^{(d-\delta)/2} \bigr). 
\]
	If in addition $\sum_j j \theta_j / j^{d/\delta} <\infty$, the same estimate holds with exponent $\min (d/2, d-\delta)$ instead of $(d-\delta)/2$. 
\end{lemma}

As we shall see later, the additional condition means that $M$ has normal fluctuations of order $L^{d/2}$; if the confining potential is quite shallow, $\delta >d/2$, our large deviations estimate kicks in at fluctuations of the order of $L^\delta$ only.

\begin{proof}
Recall that the cumulant generating function of a Poisson variable $N\sim \mathrm{Poiss}(\lambda)$ is $\log \bbE\e{t N} = \lambda(\e{t} - 1)$. Let $t= t_L \in \bbR$ with $t _L^2 = O(L^{d-\delta})$. 
We have 
\be
	\begin{aligned}
		\log \bbE_L^{z_{\rm c}}\bigl[\e{t_L (\frac{1}{L^d} M -\rho_{\rm c}^L)} \bigr]  
			& = \sum_{j\geq 1} \sum_{x\in \bbZ^d\backslash \{0\}}  
				\log \bbE_L^{z_{\rm c}}\bigl[\e{j t_{L} \frac1{L^{d}} ( R_{xj}(\omega) - \bbE_L^{z_{\rm c}} R_{xj})} \bigr]\\
			& = \sum_{j\geq 1} \sum_{x\in \bbZ^d\backslash \{0\}} 
				\frac{\theta_j}{j}\e{-j V(x/L)} \bigl( \e{j t_L /L^d} - 1 - j \frac{t_L}{L^d} \bigr).  					
	\end{aligned}
	\ee  
	Using $|\e{s} - 1 - s| \leq \frac12 s^{2} \e{|s|}$, the interior sum is bounded by 
	\be \label{eq:interior-sum}
		\frac{1}{2} \theta_j \frac{1}{L^d}  \sum_{x\in \bbZ^d\backslash \{0\}}  \e{- \frac{1}{2} j V(x/L)} \Bigl[  \frac{j t_L^2}{L^d} \e{- 
		j ( \frac{1}{2}  V(x/L) - \frac{1}{L^{d}} t_L)}\Bigr].  
	\ee
Because of Assumption \ref{ass 2} (i), there exists $b'>0$ such that $V(x/L) \geq  b' L^{-\delta}$ for all non-zero $x\in \bbZ^{d}$. 
Since $t_L \ll L^{d-\delta}$, the square bracket is bounded, for large $L$, by 
	\be
		  j \frac{t_L^2}{L^d} \exp\bigl(- \frac{b'}{ 4 L^\delta}j \bigr) \leq \frac{t_L^2}{L^d} \times \frac{4 L^\delta}{b' {\rm e}}.	
	\ee
Then there is a constant $C'>0$ such that
	\be
	\label{Pfff... encore une borne}
		\Bigl| \log \bbE_L^{z_{\rm c}}\bigl[\e{t_L (\frac{1}{L^d} M -\rho_{\rm c}^L)} \bigr] \Big| \leq C' \frac{t_L^2}{L^{d-\delta}}  \sum_{j\geq 1} \frac{\theta_j}{j^{d/\delta}} = O(1).
	\ee
Using Markov's inequality,
	\be
		\bbP_L^{z_{\rm c}} \Bigl( \tfrac1{L^{d}} M - \rho_{\rm c}^{L} \geq B \Bigr) 
			\leq \e{- t_L B} \bbE_L^{z_{\rm c}}\bigl[\e{t_L (\frac{1}{L^d} M - \rho_{\rm c}^L) } \bigr],
	\ee
and a similar bound for the probability that $\frac1{L^{d}} M - \rho_{\rm c}^{L} \leq - B$ (use a negative $t_L$).

If $\sum_j j \theta_j / j^{d/\delta}<\infty$, we choose $t_L = O(L^{d-\delta})$ to bound the square bracket in \eqref{eq:interior-sum} by a constant times $j$. The right side in \eqref{Pfff... encore une borne} is then replaced by ${\rm const} (t_{L}^{2}/L^{d}) \sum j\theta_{j} / j^{d/\delta}$ and it is bounded when $t_{L} = O(L^{d/2})$.
\end{proof}

\begin{proof}[Proof of Theorem~\ref{thm H0}]
We have  
\be
	\bbP_L^{z_{\rm c}}\Bigl( \sum_{j=1}^n j R_{0j} = \ell \Bigr)  = h_\ell\exp\Bigl( - \sum_{j=1}^n \frac{\theta_j}{j}\Bigr)
\ee
so that
\begin{align} 
	&\bbP_{L,n}(H_{0}=\ell) =  \frac{ \bbP_L^{z_{\rm c}}(\sum_{j=1}^n j R_{0j} = \ell) \bbP_L^{z_{\rm c}}( M = n-\ell)}{ \bbP_L^{z_{\rm c}}(\sum_{j=1}^n j R_{0j}+ M =n ) },\label{eq:conditioned-critical}  \\
	&\bbE_{L,n} \bigl[ \e{t (n- H_0)}\bigr] = \frac{\bbE_L^{z_{\rm c}}[ h_{n - M}  \e{t M} ]} {\bbE_L^{z_{\rm c}} [h_{n-M}] }. \label{eq:conditioned-fourier}
\end{align}
We use the convention that $h_{m}=0$ when $m<0$.

We show that $\bbE_{L,n} \bigl[ \exp(\frac t{L^{d}} (n- H_0)) \bigr] \to \e{t \rho_{\rm c}}$ for any $t$. Convergence in distribution follows from the pointwise convergence of the Laplace transform and we get the claim of the theorem. It is enough to show that for any $\eps>0$, we have
\be
\frac{\bbE_L^{z_{\rm c}}[ h_{n - M} \mathbf{1}_{|M/L^{d} - \rho_{\rm c}^{L}| > \eps}]} {\bbE_L^{z_{\rm c}} [h_{n-M} ] } \to 0.
\ee
We have
\be
\begin{split}
\frac{\bbE_L^{z_{\rm c}}[ h_{n - M} \mathbf{1}_{|M/L^{d} - \rho_{\rm c}^{L}| > \eps}]} {\bbE_L^{z_{\rm c}} [h_{n-M} ] } &\leq \frac{\max_{j \leq n} h_{j}}{\min_{j \leq n} h_{j}} \; \frac{\bbP_{L}^{z_{\rm c}}(|\frac1{L^{d}} M - \rho_{\rm c}^{L}| > \eps)}{\bbP_{L}^{z_{\rm c}}(M \leq n)} \\
&\leq \frac{\max_{j \leq n} h_{j}}{\min_{j \leq n} h_{j}} \, C \e{-c\eps L^{(d-\delta)/2}}.
\end{split}
\ee
The last inequality follows from Lemma~\ref{lem:m-laplace}. The last term vanishes indeed as $L\to\infty$.
\end{proof}

\subsection{Fluctuations of the condensate}

We now study the fluctuations of the condensate. The goal is to find the correct scaling $\alpha$ and the limiting random variable $X$ such that
\be
\frac{H_{0} - L^{d} (\rho - \rho_{\rm c}^{L})}{L^{\alpha}} \todi X.
\ee
For simplicity, we fix the potential $V(x) = c \|x\|^{\delta}$ in this subsection.
It turns out that the phase diagram of the fluctuations is very rich. It is pictured in Fig.\ \ref{fig:fluct}. We only provide partial results, see the regions in dark colors. The lightly colored region is left to future studies; it would certainly be interesting to know what happens there.

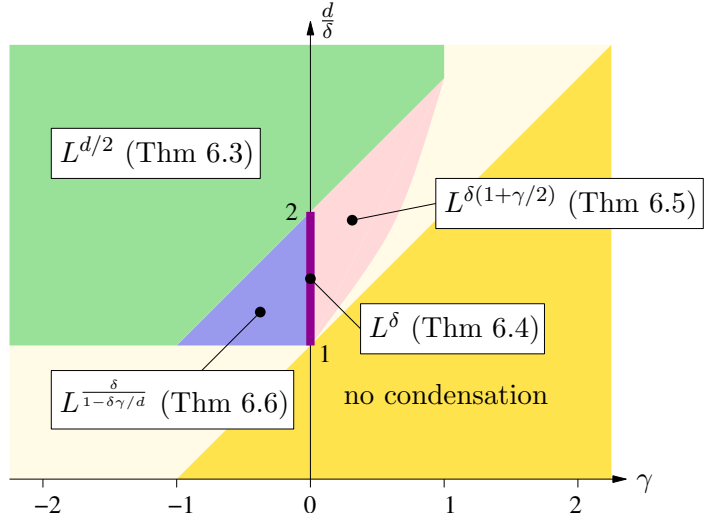
\begin{figure}[htb]
\begin{center}
\begin{picture}(0,0)%
\includegraphics{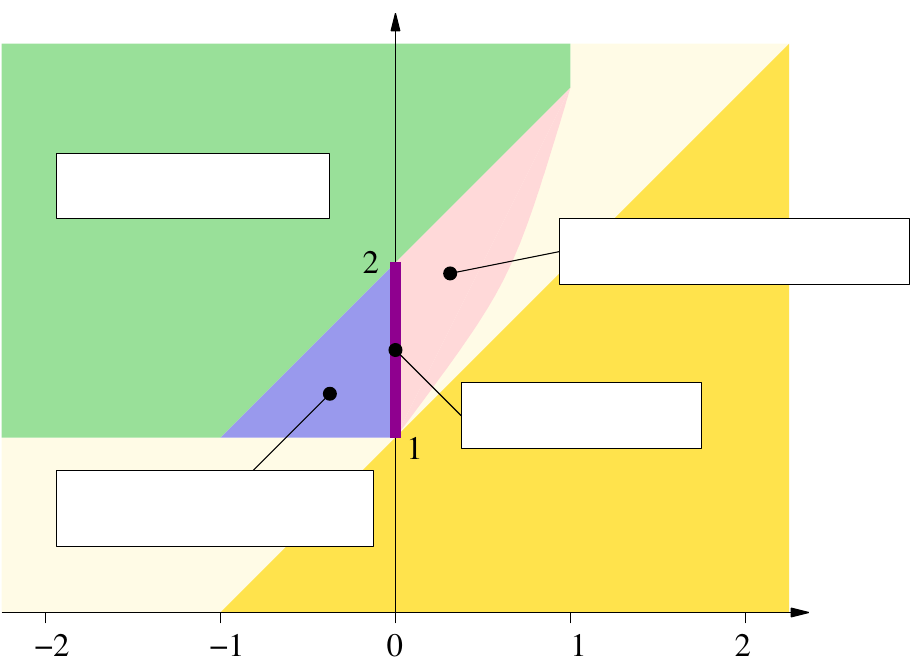}
\end{picture}%
\setlength{\unitlength}{2763sp}%
\begingroup\makeatletter\ifx\SetFigFont\undefined%
\gdef\SetFigFont#1#2#3#4#5{%
  \reset@font\fontsize{#1}{#2pt}%
  \fontfamily{#3}\fontseries{#4}\fontshape{#5}%
  \selectfont}%
\fi\endgroup%
\begin{picture}(6249,4494)(2689,-6661)
\put(5701,-5686){\makebox(0,0)[lb]{\smash{{\SetFigFont{11}{13.2}{\rmdefault}{\mddefault}{\updefault}{\color[rgb]{0,0,0}no condensation}%
}}}}
\put(8326,-6436){\makebox(0,0)[lb]{\smash{{\SetFigFont{11}{13.2}{\rmdefault}{\mddefault}{\updefault}{\color[rgb]{0,0,0}$\gamma$}%
}}}}
\put(5476,-2311){\makebox(0,0)[lb]{\smash{{\SetFigFont{11}{13.2}{\rmdefault}{\mddefault}{\updefault}{\color[rgb]{0,0,0}$\frac d\delta$}%
}}}}
\put(5926,-5086){\makebox(0,0)[lb]{\smash{{\SetFigFont{11}{13.2}{\rmdefault}{\mddefault}{\updefault}{\color[rgb]{0,0,0}$L^\delta$ (Thm \ref{thm:ewens-fluct})}%
}}}}
\put(3151,-5761){\makebox(0,0)[lb]{\smash{{\SetFigFont{11}{13.2}{\rmdefault}{\mddefault}{\updefault}{\color[rgb]{0,0,0}$L^{\frac\delta{1-\delta\gamma/d}}$ (Thm \ref{thm:alg-neg-fluct})}%
}}}}
\put(6601,-3961){\makebox(0,0)[lb]{\smash{{\SetFigFont{11}{13.2}{\rmdefault}{\mddefault}{\updefault}{\color[rgb]{0,0,0}$L^{\delta(1+\gamma/2)}$ (Thm \ref{thm:alg-fluct})}%
}}}}
\put(3151,-3511){\makebox(0,0)[lb]{\smash{{\SetFigFont{11}{13.2}{\rmdefault}{\mddefault}{\updefault}{\color[rgb]{0,0,0}$L^{d/2}$ (Thm \ref{thm clt})}%
}}}}
\end{picture}%
\end{center}
\caption{\label{fig:fluct}  
\small
Phase diagrams of the fluctuations of the condensate for the weights $\theta_j = j^\gamma$ and the potential $V(x) = c\|x\|^\delta$. Results for the dark regions are stated in Theorems~\ref{thm clt} to~\ref{thm:alg-neg-fluct}. For $\frac{d}{\delta}\leq \gamma+ 1$, there is no condensation. The light region remains to be investigated.} 
\end{figure} 

The first result deals with normal fluctuations. Let
\be
\label{def sigma}
\sigma_{\rm c}^2 = \sum_{j\geq 1} j\theta_j \int_{\bbR^d} \e{-j V(x)} \dd x.
\ee
The first result holds in the case where the variance above is finite.

\begin{theorem}[Central limit theorem for $H_{0}$]
\label{thm clt}
Assume that $V(x) = \|x\|^{\delta}$, $\rho>\rho_{\rm c}$ and $\sum_{j\geq 1} j \theta_j/j^{d/\delta}<\infty$. In addition, assume that there exist $C,c>0$, $a<\min(\frac12, 1-\frac\delta d)$, and $b>a+\max(\frac12, \frac\delta d)$ such that for all $n$,
\[
\begin{split}
&C^{-1} \e{- c n^a} \leq h_{n} \leq C \e{c n^a}, \\
&\lim_{n\to\infty} \max_{m\in [-n^{b}, n^{b}]} \frac{h_{n+m}}{h_{n}} = 1.
\end{split}
\]
Then, as $L,n \to \infty$ with $n = \rho L^{d}$,
\be
	\frac{H_{0} - L^{d} (\rho - \rho_{\rm c}^{L})}{\sigma_{\rm c} L^{d/2}} \todi \caN(0,1).
\ee
\end{theorem}

The conditions on the $h_{n}$s can be easily verified for many given $(\theta_{j})$ with the help of Theorem \ref{thm buv} and Eq.~\eqref{eq:hn-alg-growth}. In the case of stretched exponential weights $\theta_{j} = \e{-j^{\alpha}}$, where $h_{n} \approx \e{-n^{\alpha}}$, we have $h_{n+m}/h_{n} \approx \e{-\alpha m/n^{1-\alpha}}$ so we need $b<1-\alpha$. This is possible only if $\alpha < \frac12 \min(\frac12, 1-\frac\delta d)$. For larger $\alpha$, or for weights $\theta_{j} = j^{\gamma}$ with large $\gamma$, we expect a deterministic shift in the average occupation number of the origin, small compared to $n$ but  large compared to $\sqrt{n}$,  and then normal fluctuations around the shifted mean; compare with Theorem~\ref{thm:homogen} (c). 

\begin{proof} This is similar to Theorem \ref{thm H0}. Since $H_{0} = \rho L^{d} - M$, it is enough to show that
\be
\bbE_{L,n}\Bigl( \exp\Bigl\{ \ii u \frac{M - L^{d} \rho_{\rm c}^{L}}{\sigma_{\rm c} L^{d/2}} \Bigr\} \Bigr) \to \e{-u^{2}/2}.
\ee
As in Eq.\ \eqref{eq:conditioned-fourier}, and again using the convention $h_{m}=0$ when $m<0$, we have
\be
\label{it works!}
\begin{split}
\bbE_{L,n}&\Bigl( \exp\Bigl\{ \ii u \frac{M - L^{d} \rho_{\rm c}^{L}}{\sigma_{\rm c} L^{d/2}} \Bigr\} \Bigr) = \frac{\bbE_{L}^{z_{\rm c}} \Bigl( h_{n-M} \exp\Bigl\{ \ii u \frac{M - L^{d} \rho_{\rm c}^{L}}{\sigma_{\rm c} L^{d/2}} \Bigr\} \Bigr)}{\bbE_{L}^{z_{\rm c}}(h_{n-M})} \\
&= \bbE_{L}^{z_{\rm c}}\Bigl( \exp\Bigl\{ \ii u \frac{M - L^{d} \rho_{\rm c}^{L}}{\sigma_{\rm c} L^{d/2}} \Bigr\} \Bigr) + \bbE_{L}^{z_{\rm c}}\Bigl( \Bigl[ \frac{h_{n-M}}{\bbE_{L}^{z_{\rm c}}(h_{n-M})} - 1 \Bigr] \exp\Bigl\{ \ii u \frac{M - L^{d} \rho_{\rm c}^{L}}{\sigma_{\rm c} L^{d/2}} \Bigr\} \Bigr).
\end{split}
\ee
We have
\be
 	\bbE_L^{z_{\rm c}}\Bigl[ \exp\Bigl(\mathrm{i} u \frac{M- \rho_{\rm c}^L L^d}{\sigma_{\rm c}L^{d/2} }\Bigl)\Bigr] 
	= \exp \biggl\{ \sum_{j\geq1} \frac{\theta_{j}}j \sum_{x \in \bbZ^{d} \setminus \{0\}} \e{-j V(x/L)} \Bigl[ \e{\ii j \frac u{\sigma_{\rm c} L^{d/2}}} - 1 - \ii j \tfrac u{\sigma_{\rm c} L^{d/2}} \Bigr] \biggr\}. 
\ee
The absolute value of the square bracket is less than ${j^{2} u^{2}}/{\sigma_{\rm c}^{2} L^{d}}$. From our assumptions and dominated convergence, we get
\be
\lim_{L\to\infty} \bbE_{L,n}\Bigl( \exp\Bigl\{ \ii u \frac{M - L^{d} \rho_{\rm c}^{L}}{\sigma_{\rm c} L^{d/2}} \Bigr\} \Bigr) = \e{-u^{2}/2}.
\ee
It remains to verify that the second term in \eqref{it works!} vanishes in the limit $L\to\infty$.

Let $\alpha =\min(\frac d2, d-\delta)$ and $\nu < \alpha - ad$.
As in the proof of Theorem \ref{thm H0}, it follows from Lemma~\ref{lem:m-laplace} that
\be
\bbE_{L}^{z_{\rm c}}\Bigl( \Bigl[ \frac{h_{n-M}}{\bbE_{L}^{z_{\rm c}}(h_{n-M})} - 1 \Bigr] {\mathbf 1}_{|\frac1{L^{d}} M - \rho_{\rm c}^{L}| > L^{-\nu}} \Bigr) \leq {\rm const} \e{2cn^{a} - c L^{\alpha-\nu}} \to 0.
\ee
We also have
\be
\bigl( 1 + o(1) \bigr) \min_{m, |m - L^{d} \rho_{\rm c}^{L}| < L^{d-\nu}} h_{n-m} \leq \bbE_{L}^{z_{\rm c}}(h_{n-M}) \leq \bigl( 1 + o(1) \bigr) \max_{m, |m - L^{d} \rho_{\rm c}^{L}| < L^{d-\nu}} h_{n-m}.
\ee
We choose $\nu>d(1-b)$, which is possible by the assumptions of the theorem. Then $\bbE_{L}^{z_{\rm c}}(h_{n-M}) = h_{n-\rho_{\rm c}^{L} L^{d}} (1+o(1))$ and the second term in \eqref{it works!} vanishes in the limit $L\to\infty$.
\end{proof}

The central limit theorem applies to the ideal Bose gas in dimensions $d \geq 5$. It was noted by Buffet and Pul\'e that the fluctuations were non-normal in $d=3$, of order $L^{2}$ \cite{BP}. We propose here three theorems about non-normal fluctuations larger than $L^{d/2}$. They always apply to situations where the variance $\sigma_{\rm c}^{2}$ defined in \eqref{def sigma} is infinite. In order to guess the scale of the fluctuations, we observe that
\be
\label{var M}
\bbE_{L}^{z_{\rm c}} \Bigl( \bigl( M - \bbE_{L}^{z_{\rm c}} M \bigr)^{2} \Bigr) = \sum_{j\geq1} j \theta_{j} \sum_{x\in\bbZ^{d}\setminus\{0\}} \e{-j V(x/L)} = O(L^{d+\delta}).
\ee
This can be shown by controlling the sum over $j$ as in the proof of Lemma \ref{lem:m-laplace}.

The first theorem deals with the Ewens case where the weights $\theta_{j}$ are constant. The critical density is finite and the variance infinite when $1< \tfrac{d}{\delta} - \gamma \leq 2$. Let $(G_{x})_{x\in\bbZ^{d}}$ be i.i.d.\ Gamma variables with parameters $(\theta,1)$ (that is, the probability density function is $\Gamma(\theta)^{-1} t^{\theta-1} \e{-t}$ for $t\geq0$).

\begin{theorem} \label{thm:ewens-fluct}
Suppose that $V(x) = c\|x\|^\delta$ with $\frac d2 < \delta < d$, and that $\theta_{j}\equiv\theta$ is constant. Then if $\rho>\rho_{\rm c}$, we have
	\[
		\frac{H_0 - (\rho-\rho_{\rm c}^L) L^d}{L^{\delta}} 
			\todi \sum_{x\in \bbZ^d\setminus \{0\}} \frac{\theta- G_x}{\|x\|^{\delta}}. 
	\]
\end{theorem} 

When $\theta=1$, $G_x$ is exponential with mean $1$ and we recover the result from~\cite{BP,CD}. The limiting random variable of Theorem \ref{thm:ewens-fluct} requires some explanations (see \cite{CD} for a closely related situation). Each $G_{x}$ has mean $\theta$. Consider
\be
Y_{\Lambda} = \sum_{x\in\Lambda} \frac{\theta- G_x}{\|x\|^{\delta}}.
\ee
Since $2\delta>d$ and $\sum_{x\in \bbZ^d \backslash \{0\}} \|x\|^{-2\delta} <\infty$, a standard theorem applies that concerns random series \cite[Chapter 2.5]{durrett-proba-book}: For every increasing sequence $(\Lambda_n)_{n\in \bbN}$ of domains with $\cup_{n} \Lambda_n = \bbZ^d$, the limit $Y(\omega) = \lim_{n\to \infty} Y_{\Lambda_n}(\omega)$ exists almost surely.  The dominated convergence theorem further shows that the limiting random variable has characteristic function 
\be
\bbE(\e{\ii u Y}) = \exp\biggl\{ \theta \sum_{x\in\bbZ^d\backslash \{0\}} \Bigl[ \frac{\ii u}{\|x\|^{\delta}} - \log \Bigl( 1 + \frac{\ii u}{\|x\|^{\delta}} \Bigr) \Bigr] \biggr\}.
\ee
It follows in particular that the law of $Y$ is independent of the precise choice of the sequence of domains. The infinite sum in Theorem \ref{thm:ewens-fluct} is by definition a variable equal to $Y$ in law.

\begin{theorem}\label{thm:alg-fluct} 
	Suppose that $V(x) = \|x\|^\delta$, $\theta_j = j^\gamma$ with $\gamma>0$, and that $(1+\gamma) (1+\frac12 \gamma) < \frac{d}{\delta} < \gamma+2$. Then if $\rho>\rho_{\rm c}$, we have
	\[
		\frac{H_0 - (\rho-\rho_{\rm c}^L) L^d}{L^{\delta(1+\gamma/2)}} 
			\todi \Bigl( \Gamma(\gamma+2) \sum_{x\in \bbZ^d\backslash \{0\}} \frac{1}{\|x\|^{\delta(\gamma+2)}}\Bigr)^{1/2} \mathcal{N}(0,1).
	\]
\end{theorem}

Next, let $b_{\gamma,\delta,d}$ be equal to
\be
b_{\gamma,\delta,d} = -\frac{\Gamma(1+\gamma)}{|\gamma|} \int_{\bbR^d} \Bigl( \bigl( 1 + \|x\|^\delta\bigr)^{|\gamma|} - \|x\|^{\delta |\gamma|} - |\gamma| \, \|x\|^{-\delta(1+\gamma)}\Bigr) \dd x.
\ee
The integrand behaves like $\|x\|^{-\delta(2 + \gamma)}$ for large $\|x\|$ and like $\|x\|^{-\delta (1+\gamma)}$ for small $\|x\|$, so the integral is well-defined. It follows from $(1+\|x\|^{\delta})^{|\gamma|} \leq \|x\|^{\delta |\gamma|} (1 + |\gamma| \, \|x\|^{-\delta})$ that $b_{\gamma,\delta,d}$ is positive. We define $Z$ to be the stable random variable whose Laplace transform for $t\geq0$ is
	\be \label{eq:alg-neg} 
		 \bbE \e{tZ}  = \e{b_{\gamma,\delta,d} \, t^{|\gamma| + d/\delta}}.
	\ee

\begin{theorem} \label{thm:alg-neg-fluct}
	Suppose that $V(x) = \|x\|^\delta$, $\theta_{j} = j^{\gamma}$ with $\gamma<0$, and that $1 < \frac d\delta < \gamma+2$. Then if $\rho>\rho_{\rm c}$, we have
\[
\frac{H_0- (\rho-\rho_{\rm c}^L) L^d}{L^\alpha} \todi Z,
\]
where $\alpha = \frac{\delta}{1 - \delta\gamma / d}$.
\end{theorem}

\begin{remark} 
	If $\delta \geq d$, an analogous result for the fluctuations of $M$ in the grand-canonical ensemble at $z=z_{\rm c}$ still holds, but  we can no longer use Lemma~\ref{lem:m-laplace} in order to pass from the grand-canonical to the canonical ensemble. We suspect that Theorem~\ref{thm:alg-neg-fluct} can be shown with an adequate replacement of Lemma~\ref{lem:m-laplace} (e.g., algebraic deviations bounds based on finite moments), but leave the proof or disproof as an open problem.
\end{remark}

We prove these theorems by showing convergence of Laplace or Fourier transforms. Let
\be
I_{s}(z) = \sum_{j\geq1} j^{s} \theta_{j} z^{j}.
\ee
We have
\begin{align}
	&\log \bbE_L^{z_{\rm c}} \Bigl[ \e{ t \bigl( M - \bbE_L^{z_{\rm c}} M\bigr)} \Bigr] 
		=  \sum_{x\in \bbZ^d\backslash \{0\}} \log \bbE_L^{z_{\rm c}} \Bigl[ \e{ t \bigl( H_x- \bbE_L^{z_{\rm c}} H_x\bigr)} \Bigr], \label{eq:mhx} \\
	&\log \bbE_L^{z_{\rm c}} \Bigl[ \e{ t \bigl( H_x- \bbE_L^{z_{\rm c}} H_x\bigr)} \Bigr] = 
 I_{-1}\bigl(\e{t- V(x/L)} \bigr) -  I_{-1}\bigl( \e{- V(x/L)} \bigr)  - t  I_0 \bigl( \e{- V(x/L)}\bigr).	 \label{eq:hxi}
\end{align} 
Combining Eq.~\eqref{eq:hxi} with asymptotic expansions of $I_{-1}(\e{-\mu})$, $I_0(\e{-\mu})$ around $\mu=0$, we obtain two lemmas on single-site fluctuations.

\begin{lemma} \label{lem:ewens-fluct} 
	Suppose that $V(x) = \|x\|^\delta$ and that $\theta_j = \theta$ for all $j$. 
	Then under $\bbP_L^{z_{\rm c}}$, for every fixed $x\in \bbZ^d\backslash \{0\}$, we have
	\[		
 \frac{ H_x - \bbE_L^{z_{\rm c}} H_x} {L^\delta}\todi  \frac{1}{\|x\|^\delta} (\theta - G_x), 
	\]
	where $G_x$ is a Gamma random variable with parameter $(\theta,1)$.
 \end{lemma}

\begin{proof} 
For $\theta_j\equiv \theta$ we have $I_{-1}(z) = - \theta \log(1- z)$. 
As $\mu\to 0$,
\be
\begin{aligned}
	&I_{-1} (\e{-\mu}) = - \theta \log (1- \e{-\mu}) =  - \theta \log \mu + \theta \frac{\mu}{2} + O(\mu^2) \\
	&I_0(\e{-\mu}) =  \frac{\theta}{1-\e{-\mu}} = \theta \mu^{-1} + o(\mu^{-1})
\end{aligned}
\ee 
hence as long as $t\to 0$, $t<\mu$, 
\be
	 I_{-1}(\e{t-\mu}) - I_{-1}(\e{-\mu}) - t I_0(\e{-\mu}) = - \theta \log \bigl(1- \tfrac{t}{\mu} \bigr) + \theta \tfrac{t}{\mu} + o(1).
\ee
Applying this to $\mu= V(x/L)$ and  $t= s L^{-\delta}$, we get
\be
	\log \bbE_L^{z_{\rm c}} \Bigl[ \e{ s L^{-\delta} \bigl( H_x - \bbE_L^{z_{\rm c}} H_x\bigr)} \Bigr]
	=  - \theta  \Bigl( \log \bigl( 1 - \tfrac{s}{\|x\|^\delta} \bigr) -
 \tfrac{s}{\|x\|^\delta} \Bigr) +o(1).
\ee
On the other hand,
\be 
\log\bbE \bigl[ \e{s(\theta-G)}\bigr] = \theta \bigl( s - \log(1+s) \bigr).
\ee
The result follows with the help of the usual continuity theorems (convergence of Laplace transforms implies convergence in distribution).
\end{proof}

\begin{lemma} \label{lem:alg-fluct} 
	Suppose that $V(x) = \|x\|^\delta$ and $\theta_j = j^\gamma$ with $\gamma>0$.
	Then 
	under $\bbP_L^{z_{\rm c}}$, for every fixed $x\in \bbZ^d\backslash \{0\}$, 
	\be
		 \frac{ H_x - \bbE_L^{z_{\rm c}}H_x }{L^{\delta(\gamma+2)/2}}\ \todi \frac{( \Gamma(\gamma+2))^{1/2}}{\|x\|^{\delta(\gamma+2)/2}} \mathcal{N}(0,1).
	\ee
\end{lemma}

\begin{proof} 
Consider first the case $\gamma\notin \bbN$, and slightly modified weights corresponding to 
\be \label{eq:gamma-dilog}
	I_0(\e{-\mu}) = \frac{\Gamma(\gamma+1)}{(1- \e{-\mu})^{\gamma+1}} - \Gamma(\gamma+1).
\ee
The corresponding weights satisfy $\theta_j = j^\gamma (1+o(1))$~\cite{ercolani-ueltschi}. 
We note the identity 
\be 
	  I_{-1}(\e{t-\mu}) - I_{-1}(\e{-\mu})- tI_0(\e{-\mu})  =  \int_0^t \bigl( I_0(\e{s-\mu}) - I_0(\e{-\mu})\bigr) \dd s.
\ee
For $\mu \to 0$, we have 
\be \label{eq:izero}
	I_0(\e{-\mu})  = \Gamma(\gamma+1)\mu^{-\gamma-1} +O(\mu^{-\gamma}). 
\ee
More precisely, $I_0(\e{-\mu})$ has an asymptotic expansion of the form $\mu^{-\gamma-1} \sum_{n\geq 0} a_n \mu^n$. As $t,\mu \to 0$ with $t = o(\mu)$, 
\be
\begin{aligned} 
	 \int_0^t \Bigl( (\mu- s)^{-\gamma- 1} - \mu^{-\gamma- 1}\Bigr) \dd s 
		& = \tfrac{1}{\gamma} \Bigl( (\mu - t)^{-\gamma} - \mu^{-\gamma}\Bigr) - t \mu^{-\gamma - 1}  \\
	& = \bigl(1+o(1)\bigr) \tfrac{1}{2} (\gamma+1) \mu^{-\gamma-2} t^2.  
\end{aligned} 
\ee
The other terms of the asymptotic expansion --- which can be pushed to arbitrarily high order ---  can be evaluated in a similar way and are found to give contributions of smaller order. 
Choosing $t= s L^{- \delta(\gamma+2)/2}$ and  $\mu = \|x\|^\delta / L^\delta$, we get 
\be
	\lim_{L\to\infty} \log \bbE_L^{z_{\rm c}} \Bigl[ \exp\biggl( \frac{s (H_x- \bbE_L^{z_{\rm c}} H_x) }{L^{\delta(\gamma+2)/2}}\biggr) \Bigr] = 
	 \tfrac{1}{2}  \Gamma(\gamma+2) s^2 \|x\|^{- \delta(\gamma+2)}
\ee 
and we recognize the Laplace transform of a normal random variable. This proves the lemma for the modified weights from~\cite{ercolani-ueltschi}. For the original weights $\theta_j = j^\gamma$, we recognize in $I_{-1}(z)$ a \emph{dilogarithm}, and the results still apply because the dilogarithm and the function~\eqref{eq:gamma-dilog} have similar asymptotic expansions~\cite[Chapter VI.8]{FS}.
\end{proof} 

Now we come to the proofs of the theorems. The proof of Theorem~\ref{thm:ewens-fluct} 
is based on Lemma~\ref{lem:ewens-fluct} and is analogous to the proofs in~\cite{BP,CD} and to the proof of Theorem~\ref{thm:alg-fluct}, and it is therefore omitted. In contrast with algebraically growing weights no additional condition on $\gamma$ is needed in order to go from the grand-canonical to the canonical ensemble. This is because $h_n = \theta(\theta+1) \cdots(\theta+n-1) /n!$ is a slowly varying function of $n$. 

\begin{proof}[Proof of Theorem~\ref{thm:alg-fluct}]
We treat first the fluctuations of $M$ in the grand-canonical ensemble at $z=z_{\rm c}$. 
To this aim we look at the Fourier tranform of the law of $L^{-\delta(\gamma+2) }(M - \bbE_L^{z_{\rm c}} M)$ under $\bbP_L^{z_{\rm c}}$; it is given by a sum over $x$ analogous to Eq.~\eqref{eq:mhx}. We split the sum in two parts. 
Revisiting the proof of Lemma~\ref{lem:alg-fluct}, we find that 
for suitable $c$ and $\eps$, all $\mu \leq \eps^\delta$, and all $|t| \leq \frac{1}{2} \mu$, 
\be
	\bigl| I_{-1} (\e{\ii t - \mu}) -  I_{-1} (\e{- \mu}) - \ii t I_0 (\e{-\mu}) \bigr| 
		\leq (\gamma+1) |t|^2 \mu^{-\gamma-2}.
\ee
We apply this inequality to $\mu = \|x\|^\delta / L^\delta$ with $\|x\|\leq \eps L$, and $t= s / L^{\delta(\gamma+2)}$, with $L$ large enough so that $ |s| L^{-\delta(\gamma+2)} \leq \frac{1}{2} L^{-\delta}$. Then 
\be
	\Bigl|\log \bbE_L^{z_{\rm c}} \bigl[ \e{\ii s L^{-\delta(\gamma+2)/2} (H_x- \bbE_L^{z_{\rm c}} H_x)} \bigr] \Bigr| \leq (\gamma+1) \frac{s^2}{\|x\|^{\delta(\gamma+2)}}.
\ee
We have $\delta(\gamma+2) >d$, so the right side can be summed over $x\in \bbZ^d\setminus \{0\}$. From Lemma~\ref{lem:alg-fluct} and dominated convergence, we get
\be
	\lim_{L\to \infty} \sum_{\substack{x\in \bbZ^d\backslash\{0\} \\ \|x\|\leq \eps L}} 
\log \bbE_L^{z_{\rm c}} \bigl[\e{\ii s L^{-\delta(\gamma+2)/2}(H_x - \bbE_L^{z_{\rm c}} H_x)} \bigr] = -\frac{s^2}{2} \sum_{x\in \bbZ^d\backslash \{0\}} \frac{\Gamma(\gamma+2) }{\|x\|^{\delta(\gamma+2)}}, 
\ee
which is the logarithm of the Fourier transform of the limiting normal random variable. It remains to show that the sum over $\|x\|\geq \eps L$ does not contribute. Notice that 
\be
	\sum_{j\geq 1} j \theta_j \int_{\|x\|>\eps} \e{- j V(x)} \dd x <\infty. 
\ee 
Arguments similar to the proof of Lemma~\ref{lem:m-laplace} show that
\be
	 \frac{1}{L^d} \bbE_L^{z_{\rm c}} \Bigl[ \sum_{\substack{x\in \bbZ^d \\ \|x\|\geq \eps L}} \bigl( H_x - \bbE_L^{z_{\rm c}} H_x \bigr)^2 \Bigr] = O(1).  
\ee
Write $M_{\eps}$ for the sum of the $H_x$ over $\|x\|\geq \eps L$.  We have, for every $\kappa>0$, 
\be
	\bbP_L^{z_{\rm c}} \Bigl( \frac{M_\eps}{L^{\delta(\gamma+2)/2 }} \geq \kappa \Bigr) 
		\leq \kappa^{-2} L^{-\delta(\gamma+2)} \bbE_L^{z_{\rm c}} \Bigl[ (M_\eps - \bbE_L^{z_{\rm c}} M_\eps)^2 \Bigr] = O(L^{d- \delta(\gamma+2)}) \to 0.   
\ee
It follows that $\bbE_L^{z_{\rm c}}[\exp( \ii s (M_\eps - \bbE_L^{z_{\rm c}} M_\eps) L^{-\delta(\gamma+2)/2})] \to 1$: the sites $\|x\|\geq 1$ do not contribute to the Fourier transform. Thus we have proven that $M$ fluctuates as $L^{\delta(\gamma+2)/2}$ times the limiting Gaussian variable. We conclude with the help of Eq.~\eqref{eq:conditioned-fourier} as in the proof of Theorems~\ref{thm H0} and Theorems~\ref{thm clt}. Let  $\eps_L = L^\beta$. The condition $\beta + \frac{1}{2} (d+\delta) > d \frac{\gamma}{\gamma+1}$ ensures that $|M- \rho_{\rm c}^L L^d|\geq \eps_L L^{(d+\delta)/2}$ does not contribute; the condition $\delta (1+\frac{1}{2}\gamma) \leq \beta + \frac{1}{2} (d+\delta) < \frac{d}{\gamma+1}$ ensures that the $h_n$ is approximately constant on the scale of fluctuations and in the interval $|M - \rho_{\rm c}^L L^d|\leq \eps_L L^d$. Such a $\beta$ exists in the region of parameters considered here. 
\end{proof}

\begin{proof} [Proof of Theorem~\ref{thm:alg-neg-fluct}]
We proceed as in the proofs of Theorems~\ref{thm H0} and Theorems~\ref{thm clt}. Observe that by Theorem~\ref{thm buv}, $h_n = (1+o(1))\theta_n /n = (1+o(1)) n^{-|\gamma| -1}$ is a regularly varying function of $n$. We can invoke Eq.~\eqref{eq:conditioned-fourier} and prove that
\be
\label{ca devrait converger}
\bbE_{L}^{z_{\rm c}} \Bigl( \e{-\frac t{L^{\alpha}} (M - \bbE_{L}^{z_{\rm c}} M)} \Bigr) \longrightarrow \bbE( \e{tZ} )
\ee
for each fixed $t$. Using Eqs~\eqref{eq:mhx} and \eqref{eq:hxi}, we have
\be
\label{ca avance}
\log \bbE_{L}^{z_{\rm c}} \Bigl( \e{-\frac t{L^{\alpha}} (M - \bbE_{L}^{z_{\rm c}} M)} \Bigr) = \sum_{x \in \bbZ^{d} \setminus \{0\}} \Bigl( I_{-1}\bigl(\e{-\frac t{L^{\alpha}} - V(x/L)}\bigr) -I_{-1}\bigl(\e{- V(x/L)}\bigr) + \frac t{L^{\alpha}} I_0\bigl(\e{- V(x/L)}\bigr) \Bigr).
\ee
Fix $\eps>0$. Just as in the proof of Theorem~\ref{thm:alg-fluct}, one can show that sites $x$ with $\|x\|\geq \eps L$ are irrelevant. Eq.~\eqref{eq:izero} stays true \cite[Chapter VI.8]{FS}.
Therefore as $\eps\to 0$, uniformly in $L$, Eq.~\eqref{ca avance} is equal to
\be
\begin{split}
&-\bigl( 1+o(1) \bigr) \frac{\Gamma(1+\gamma)}{|\gamma|} \sum_{\substack{x\in \bbZ^d\backslash \{0\} \\ \|x\|\leq \eps L}} \biggl[ \Bigl( \frac{t}{L^\alpha} + \frac{\|x\|^\delta}{L^\delta} \Bigr)^{|\gamma|} - \Bigl( \frac{\|x\|^\delta}{L^\delta} \Bigr)^{|\gamma|} - |\gamma| \frac{t}{L^\alpha} \Bigl( \frac{\|x\|^\delta}{L^\delta}\Bigr)^{|\gamma|-1} \biggr] \\
	&= -\bigl( 1+o(1) \bigr) \frac{\Gamma(1+\gamma)}{|\gamma|} \sum_{\substack{x\in \bbZ^d\backslash \{0\} \\ \|x\|\leq \eps L}} \frac{1}{L^{\alpha|\gamma|}} \biggl[ \Bigl( t + \Bigl\| \frac x{L^{1-\frac\alpha\delta}} \Bigr\|^{\delta} \Bigr)^{|\gamma|} - \Bigl\| \frac x{L^{1-\frac\alpha\delta}} \Bigr\|^{\delta |\gamma|} - |\gamma| t \Bigl\| \frac x{L^{1-\frac\alpha\delta}} \Bigr\|^{-\delta (1+\gamma} \biggr].
\end{split} 
\ee
We recognize a Riemann sum in the equation above, and this converges as $L\to\infty$ to
\be
-\frac{\Gamma(1+\gamma)}{|\gamma|} \int_{\bbR^{d}} \Bigl[ \bigl( t + \|x\|^{\delta} \bigr)^{|\gamma|} - \|x\|^{\delta |\gamma|} - |\gamma| \, t \, \|x\|^{-\delta (1+\gamma)} \Bigr] \dd x = t^{|\gamma| + d/\delta} \, b_{\gamma,\delta,d}.
\ee
This proves~\eqref{ca devrait converger}, and Theorem~\ref{thm:alg-neg-fluct}.
\end{proof}

\subsection{Size of partition elements} \label{sec:phase-diagram} 

In this section we have so far discussed the occupation number at $x=0$. We now turn to the partition structure of the condensate. Again we consider the potential $V(x)=\|x\|^{\delta}$ and the weights $\theta_j = j^\gamma$. In view of Proposition~\ref{prop:conditional} (a), and since by Theorem ~\ref{thm H0} we have $H_0/L^d\to \rho-\rho_{\rm c}$ when $\rho>\rho_{\rm c}$, it is natural to assume that the limiting structure of the partition $\lambda_0$ under $\bbP_{L,n}$ is the same as that of a single-site Gibbs partition $\nu_m$ when $m \sim (\rho-\rho_{\rm c})L^d\to \infty$. We do not wish to carry out the proof here
(see however the proof of Proposition~\ref{prop:equivalence} in Section~\ref{sec:square})
and we base our discussion on the assumption that this is true.

Recall the definition \eqref{def nu oo} of $\nu_{\infty}$, and let 
\be
\nu_{\rm macro} = \lim_{\eps \searrow 0} \limtwo{n\to\infty}{\rho L^{d}=n} \bbE_{L,n} \Bigl( \frac1n \sum_{j>\eps n} j \, R_{j} \Bigr).
\ee
be the fraction of particles in partition components of macroscopic size, i.e., of the order of  $n$. It is clear that $\nu_{\infty} \geq \nu_{\rm macro}$. There are three phases in the $(\gamma,\rho)$-plane, as illustrated in Fig.~\ \ref{fig phd}.  For $\rho<\rho_{\rm c}$, we have $\nu_\infty= \nu_{\rm macro}= 0$ by Theorem~\ref{thm trans dens}.

For $\rho>\rho_{\rm c}$, the site 0 is macroscopically occupied and the partition structure results from the Gibbs measure on partitions, $\nu_{m}$. There are three possibilities, depending on the parameter $\gamma$:
\begin{itemize}
\item[$\gamma<0$] The partition displays a unique large element that contains all the indices save a (random) finite number \cite{buv}. This gives a phase with macroscopic elements.
\item[$\gamma =0$] (Uniform random permutations or the Ewens measure.) The number of partition elements is logarithmic, they have typically a size proportional to $m$ and their joint law is asymptotically a Poisson-Dirichlet law, as 
for the lengths of cycles in uniform random permutations; see e.g.\ \cite{arratia-barbour-tavare}.
Here also, $\nu_{\rm macro} = \nu_{\infty}$.
\item[$\gamma>0$] It was proved in \cite{ercolani-ueltschi,DM} that typical elements have size of order $m^{1/(\gamma+1)}$, with Gamma distribution. Moreover the number of elements divided by   
$m^{\gamma/(\gamma+1)}$ converges in law to some constant \cite{erlihson-granovsky, ercolani-ueltschi}, and the limit behavior of the partition is described by a deterministic limit shape \cite{erlihson-granovsky-limitshapes,cipriani-zeindler}: there is a function $w(x)$ such that for all $x>0$, we have the convergence in law under $\nu_m$: 
\be
	\frac{1}{m^{\gamma/(\gamma+1)}} \# \{ j: \lambda_j \geq x m^{1/(\gamma+1)} \} \to w(x).
\ee
\end{itemize}

\begin{remark}
The Poisson-Dirichlet process appearing for $\gamma=0$ has been generalized and belongs to a 
 two-parameter family introduced in~\cite{pitman-yor}, called \emph{two-parameter Poisson-Dirichlet} or \emph{Pitman-Yor} processes. They can be constructed using subordinators or stick-breaking schemes~\cite{pitman-yor} and can appear as limits of more general Gibbs partitions where the weight is allowed to depend explicitly on the number of cycles~\cite{pitman}. 
\end{remark}

\begin{figure}[htb]
\begin{center}
\begin{picture}(0,0)%
\includegraphics{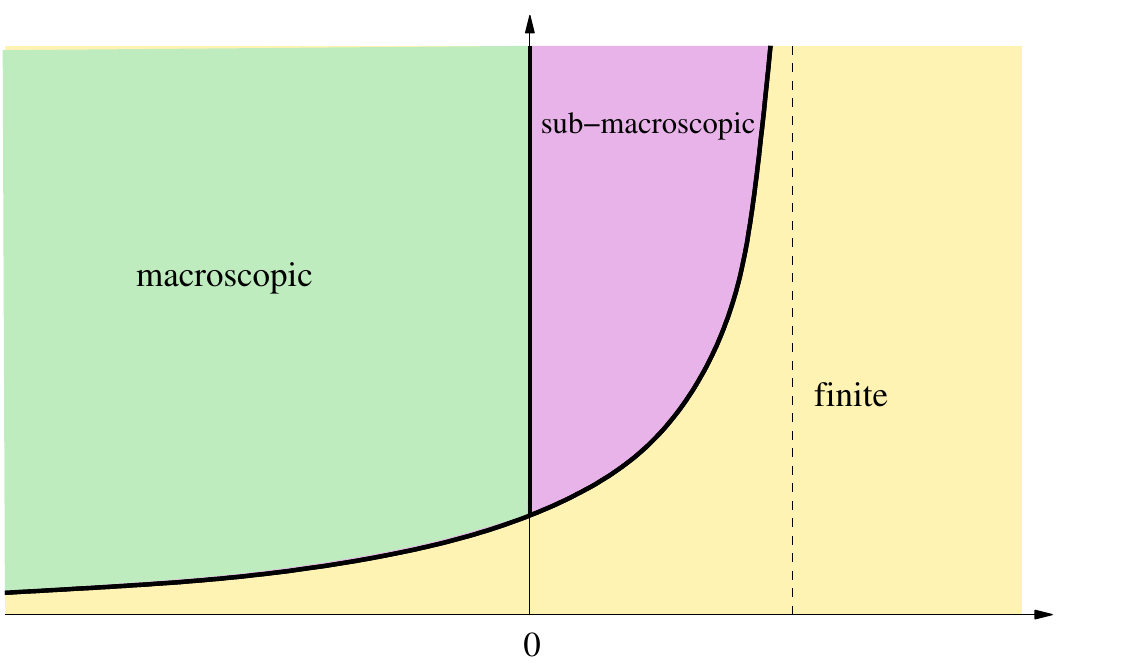}
\end{picture}%
\setlength{\unitlength}{2763sp}%
\begingroup\makeatletter\ifx\SetFigFont\undefined%
\gdef\SetFigFont#1#2#3#4#5{%
  \reset@font\fontsize{#1}{#2pt}%
  \fontfamily{#3}\fontseries{#4}\fontshape{#5}%
  \selectfont}%
\fi\endgroup%
\begin{picture}(7736,4564)(1168,-5519)
\put(6376,-5461){\makebox(0,0)[lb]{\smash{{\SetFigFont{11}{13.2}{\rmdefault}{\mddefault}{\updefault}{\color[rgb]{0,0,0}$\frac d\delta - 1$}%
}}}}
\put(2101,-3361){\makebox(0,0)[lb]{\smash{{\SetFigFont{11}{13.2}{\rmdefault}{\mddefault}{\updefault}{\color[rgb]{0,0,0}$\nu_{\rm macro} = \nu_\infty > 0$}%
}}}}
\put(4951,-2311){\makebox(0,0)[lb]{\smash{{\SetFigFont{11}{13.2}{\rmdefault}{\mddefault}{\updefault}{\color[rgb]{0,0,0}$\nu_\infty > 0$}%
}}}}
\put(4951,-2686){\makebox(0,0)[lb]{\smash{{\SetFigFont{11}{13.2}{\rmdefault}{\mddefault}{\updefault}{\color[rgb]{0,0,0}$\nu_{\rm macro} = 0$}%
}}}}
\put(8551,-5236){\makebox(0,0)[lb]{\smash{{\SetFigFont{11}{13.2}{\rmdefault}{\mddefault}{\updefault}{\color[rgb]{0,0,0}$\gamma$}%
}}}}
\put(6151,-4111){\makebox(0,0)[lb]{\smash{{\SetFigFont{11}{13.2}{\rmdefault}{\mddefault}{\updefault}{\color[rgb]{0,0,0}$\nu_{\rm macro} = \nu_\infty = 0$}%
}}}}
\put(4951,-1111){\makebox(0,0)[lb]{\smash{{\SetFigFont{11}{13.2}{\rmdefault}{\mddefault}{\updefault}{\color[rgb]{0,0,0}$\rho$}%
}}}}
\end{picture}%
\end{center}
\caption{\small Phase diagram for the family of measures with $V=\|x\|^{\delta}$ and $\theta_{j} = j^{\gamma}$. The transition line bordering the finite phase is given by the transition density $\rho = \rho_{\rm c}$ defined in \eqref{Einstein}. }
\label{fig phd}
\end{figure}
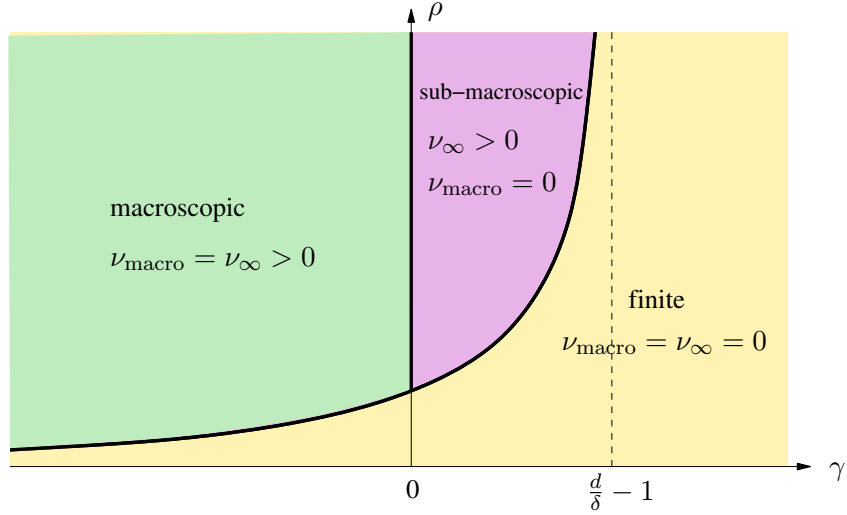

\section{Square potentials} \label{sec:square}

This last section is devoted to square traps, $\e{-V(x)} = \mathbf{1}_{(-\frac12,\frac12]^d}(x)$.
There is not much spatial structure left, as the $L^{d}$ available sites are all alike with indifferent location. But this case is actually very interesting mathematically as it involves conditioned random variables. And it is relevant in mathematical physics because it describes the stationary measure of the zero-range process, and certain equilibrium measures in droplet formations.

The results of Theorems \ref{thm trans dens} and \ref{thm infinite elements} are valid here. In the case $z_{\rm c}=1$, the transition density \eqref{Einstein} takes the simple form
\be
\rho_{\rm c} = \sum_{j\geq1} \theta_{j}.
\ee
We refer to our previous discussion of the zero-range process for background and literature behind this formula. In contrast to the confining potentials of Section \ref{sec:traps}, condensation does not take place at $x=0$ but on a random location picked uniformly at random. On that site, the partition has one large element. Instead of $H_{0}$, we consider here the following random variables:
\be
\begin{split}
&M_L = \max_{x\in \bbZ^d, \|x\|_{1} \leq \frac L2} H_x, \\
&T_L = \max\{ j \in \bbN:\ R_j \geq 1\}.
\end{split}
\ee

The main theorem of this section is due to Nagaev \cite{nagaev}, and to Armend\'ariz, Grosskinsky, and Loulakis \cite{AGL}, see Theorem \ref{thm:homogen} below. It concerns the random variable $M_L$ and we complement it with a result for $T_L$. Before this we discuss the grand-canonical measure $\bbP_{L}^{z_{\rm c}}$ and the connection with Cram{\'e}r series. At the end of the section, we comment on the relation with surface tension and the critical point of fluids. 

For concreteness we focus on algebraic and stretched exponential weights, $\theta_j= j^{-\alpha}$ with $\alpha >1$ or $\theta_j =j \exp(- j^\gamma)$ with $0<\gamma<1$. Under $\bbP_{L}^{z_{\rm c}}$ the random variables $H_{x}$ are i.i.d.\ and their distribution satisfies
\be
\bbP_L^{z_{\rm c}}(H_x=n) = h_n \exp( - \sum_{j\geq 1} \theta_j/j) = \frac{\theta_n}n (1+o(1)).
\ee
The last relation follows from Theorem~\ref{thm buv}. The mean and variance of $H_{x}$ are
\be
\begin{split}
&\bbE_{L}^{z_{\rm c}}(H_{x}) = \sum_{j\geq1} \theta_{j} = \rho_{\rm c}, \\
&\bbE_{L}^{z_{\rm c}}\bigl( [H_{x}-\bbE_{L}^{z_{\rm c}}(H_{x})]^{2} \bigr) = \sum_{j\geq1} j \theta_{j} = \sigma_{\rm c}^{2}.
\end{split}
\ee
The variance is consistent with Eq.~\eqref{var M}.

It turns out that the limit law for stretched exponential weights with fast decay ($\frac{1}{2} \leq \gamma <1$) involves a deterministic shift $\Delta_L$. In order to define it, consider first the cumulant generating function 
\be
	\varphi(t) = \log \bbE^{z_{\rm c}}_L \bigl( \e{ t H_x} \bigr) = \sum_{j\geq 1} \frac{\theta_j}{j} \bigl( \e{tj} - 1\bigr).
\ee
It is finite for $t\leq 0$. Given $\tau <0$, let 
$t_\tau$ be the unique solution of the equation $\varphi'(t) = \rho_{\rm c} + \tau$. 
Then as $\tau \nearrow 0$, the Legendre transform $\tilde\varphi(\tau) = \sup_{t\in \bbR} (t \tau - \varphi(t))$ has an  asymptotic expansion of the form
\be \label{eq:cramer} 
	- \tilde \varphi(\tau) = \varphi(t_\tau) - (\rho_{\rm c}+\tau) t_\tau  = - \frac{\tau^2}{2\sigma^2} + \tau^3 \sum_{k \geq 0} \lambda_k \tau^k.
\ee
The coefficients $\lambda_k$ are rational functions of the cumulants and they form the \emph{Cram{\'e}r series} of $H_{x}$ ~\cite{ibragimov-linnik}.  In a more general context, the series provides corrections to the central limit theorem when looking at \emph{moderate deviations} that are larger than $\sqrt{n}$ but smaller than $n$; see Chapter~7 in \cite{ibragimov-linnik}. For heavy-tailed variables, the Cram{\'e}r series diverges and we truncate it at 
$t:= \lfloor (1-\gamma)^{-1} - 2 \rfloor+1$. Let
\be \label{eq:surface}
	f_{L}(\Delta) =  \bigl( (\rho- \rho_{\rm c}) L^d  - \Delta\bigr)^\gamma +  \frac{\Delta^2}{2\sigma_{\rm c}^2 (L^d-1)} - \frac{\Delta^3}{(L^d-1)^2} \sum_{k= 0}^t \lambda_k \bigl( \frac{\Delta}{L^d-1}\bigr)^k.
\ee
It follows from the results of Nagaev~\cite{nagaev} that for $\rho>\rho_{\rm c}$ and $\frac{1}{2} \leq \gamma <1$, 
\be \label{eq:nagaev} 
	\bbP_L^{z_{\rm c}} \Bigl( \sum_{x\in \bbZ^d, \|x\|_{1} \leq L/2} H_x = \rho L^d\Bigr) 
		 = \bigl( 1+o(1)\bigr) L^d \exp \Bigl( - \min_{\Delta>0} f_{L}(\Delta) \Bigr). 
\ee
The interpretation of Eq.~\eqref{eq:nagaev} is that one of the $L^d$ sites captures almost all of the excess particles $(\rho- \rho_{\rm c}) L^d$, the remaining $L^d-1$ sites have a small excess $\Delta$ of particles, and the shift $\Delta$ minimizes the free energy penalty $f_{L}(\Delta)$.  In order to compute $\Delta$,  it is convenient to change variables and write $\Delta = a (\rho- \rho_{\rm c}) L^d $; setting the derivative equal to zero yields the equation 
\be
	\frac{\gamma \sigma_{\rm c}^2 (L^d-1)}{((\rho- \rho_{\rm c}) L^d)^{2-\gamma}} = a (1-a)^{1-\gamma} \Bigl( 1- \sigma_{\rm c}^2 \sum_{k=0}^t (k+3) \lambda_k \Bigl( \frac{(\rho- \rho_{\rm c}) L^d a}{L^d-1} \Bigr)^{k+2} \Bigr).
\ee
It has a solution $a \sim \gamma \sigma^2 L^{d (\gamma-1)} (\rho-\rho_{\rm c})^{\gamma-2}$. The minimizer $\Delta_L$ satisfies
\be
\label{Delta min}
	\Delta_L = \bigl(1+o(1)\bigr) \gamma \sigma_{\rm c}^2  (\rho-\rho_{\rm c})^{\gamma-1}  L^{d\gamma} .
\ee
When $0<\gamma<\frac{1}{2}$, the shift $\Delta_L \ll \sqrt{L^d}$ is negligible and Eq.~\eqref{eq:nagaev} is replaced by the simpler relation~\cite{nagaev}
\be \label{eq:heavy-tailed}
	\bbP_L^{z_{\rm c}} \Bigl( \sum_{x\in \bbZ^d, \|x\|_{1} \leq L/2} H_x = \rho L^d\Bigr)  \sim  L^d \bbP_L^{z_{\rm c}} \bigl( H_x= (\rho-\rho_{\rm c}) L^d\bigr) \sim  \frac{\theta_{(\rho-\rho_{\rm c})L^d}}{\rho-\rho_{\rm c}}.
\ee
(The last identity follows from Theorem~\ref{thm buv}). Eq.~\eqref{eq:heavy-tailed} also holds in the case of the algebraic weights $\theta_{j} = j^{\alpha}$ with $\alpha>1$ \cite{doney}. 

Before formulating the theorem on $M_L$ we recall the definition of $\alpha$-stable random variables. They are needed in the case of slowly decaying algebraic weights ($1<\alpha<2$). They are conveniently defined via their Laplace transform: For $t\geq0$,
\be \label{eq:laplace-stable}
	 \bbE\bigl[\e{- t Y}\bigr] =\exp\Bigl( \frac{1}{\alpha} \int_0^\infty \bigl( \e{ - t u} - 1 + t u)  \frac{\dd u}{u^{1+\alpha}}\Bigr) = \exp\Bigl( \frac{1}{\alpha} M(\alpha) t^\alpha\Bigr),
\ee
where $M(\alpha) = \int_0^\infty(\e{-y} -1 +y) y^{-1-\alpha} \dd y>0$. The random variable $Y$ has zero expectation $\bbE(Y) =0$; its right tail is $\bbP( Y\geq x) = (1+o(1)) \alpha^{-1} x^{-\alpha}$ as $x\to \infty$, and its left tail $\bbP(Y\leq -x)$ is exponentially small ~\cite{gnedenko-kolmogorov}.

Recall that $\rho_{\rm c} = \sum_{j\geq1} \theta_{j}$ and $\sigma_{\rm c}^{2} = \sum_{j\geq} j \theta_{j}$. The results from conditioned sums of random variables can be formulated in our context as follows.

\begin{theorem} [Nagaev~\cite{nagaev}, Armend{\'a}riz, Grosskinsky, Loulakis \cite{armendariz-loulakis,AGL}] \label{thm:homogen}
		Assume that $\rho>\rho_{\rm c}$.
	\begin{itemize}
		\item[(a)] If $\theta_j = j^{-\alpha}$ with $1<\alpha <2$, then 
			\[
				M_L = ( \rho- \rho_{\rm c}) L^d -  L^{d/\alpha} Y_L,
			\]
			where $Y_L$ converges in law to the $\alpha$-stable variable defined in~\eqref{eq:laplace-stable}. 	
		\item[(b)] If $\theta_j = j^{-\alpha}$ with $\alpha > 2$ or $\theta_j = \exp(- j^\gamma)$ with $0<\gamma<\frac{1}{2}$, then 
			\[
				M_L = ( \rho- \rho_{\rm c}) L^d + \sigma_{\rm c} L^{d/2}\, Y_L
			\]
			where $Y_L$ converges in law to a standard normal variable.  
		\item[(c)] If $\theta_j = j \exp(- j^\gamma)$ with $\frac{1}{2}\leq \gamma<1$, then  
			 \[
				M_L = ( \rho- \rho_{\rm c}) L^d  - \Delta_L + \sigma_{\rm c} L^{d/2}\,  Y_L 
			\]	
			where $Y_L$ converges in law to a standard normal variable, and $\Delta_L$ is the deterministic shift of order $L^{d\gamma}\geq L^{d/2}$ defined in \eqref{Delta min}. 
	\end{itemize} 
\end{theorem}

The next proposition is a variant of the non-spatial Theorem~3.2  in~\cite{buv}.

\begin{proposition}  \label{prop:equivalence}
	Suppose that $\rho>\rho_{\rm c}$ and that the weights are from one of the cases in Theorem~\ref{thm:homogen}.  Let $(N_j)_{j\geq 1}$ be independent Poisson random variables with parameters $\theta_j /j $. Then $\sum_{j\geq 1} j N_j$ is almost surely  finite, and under $\bbP_{L,n}$ we have the convergence in law
	\be
		M_L - T_L \todi  \sum_{j\geq1} j N_j.
	\ee
\end{proposition}

The proof is given at the end of this section. Proposition~\ref{prop:equivalence} implies that for all  sequences $a_L\to \infty$, 
$
	\bbP_{L,n}\bigl( M_L - T_L \geq  a_L\bigr) \to 0.
$
Therefore $T_L$ and $M_L$ have the same fluctuations, and we have proved the following:

\begin{theorem} \label{thm:homogen-partition} 
	Theorem~\ref{thm:homogen} holds true for $T_L$ as well as $M_L$. 
\end{theorem}

Thus we have proved a  limit law for the largest component of a Gibbs partition of $n$ with probability weights 
\be\label{eq:gibbs-ext}
	 \frac{1}{Z_{L,n}} \prod_{j\geq 1} \frac{1}{r_j!} \Bigl( L^d \frac{\theta_j}{j} \Bigr)^{r_j}
\ee
(see Proposition~\ref{prop:marginals}), and Eqs.~\eqref{eq:nagaev} and~\eqref{eq:heavy-tailed} provide, up to the factor $\exp (- L^d \sum_{j} \theta_j /j)$, the precise asymptotics of the partition function $Z_{L,n}$. 

From a mathematical point of view,  Theorem~\ref{thm:homogen-partition} is new\footnote{Bogachev and Zeindler prove the existence of a giant cycle of size of order $(1-\rho_{\rm c}/{\rho})n$ for a related model of random permutations, but without results on its fluctuations~\cite[Corollary 5.14]{BZ}.}  and it highlights  similarities between limit laws for the Gibbs partition~\eqref{eq:gibbs-ext} and sums of i.i.d.\ variables. These similarities might seem surprising at first, but we have given a natural explanation by providing a coupling (the measure $\bbP_{L,n}$) to  a model with  i.i.d.\ variables (the $H_x$s).

From a physical point of view, the coupling $\bbP_{L,n}$ establishes an explicit relation between stationary measures of the zero-range process on the one hand and the ideal Bose gas and classical droplet models on the other hand. The former is often used for driven systems that are \emph{out of equilibrium}, the latter belong to \emph{equilibrium} statistical mechanics. It is amusing to note, moreover, that the intricacies of the large deviations of stretched exponential variables with $\gamma \geq \frac{1}{2}$ have a natural physical interpretation in the droplet model, and the non-normal fluctuations  are related to the critical point of fluids.

Indeed, it is common to model the free energy of a droplet of volume $j$ as a bulk term proportional to $j$ plus a surface term $j^\gamma$ with $\gamma = (d-1) /d$~\cite{sator}. 
Assume that the bulk term is zero and set $\mu_{\rm sat} = 0$ (the chemical potential at which the gas saturates).
Then $p(\mu) = \sum_j (\theta_j/j) \e{\mu j}$ is the pressure of the gas, and Eq.~\eqref{eq:cramer} gives an asymptotic expansion of the free energy
\be \label{eq:feexp}
\begin{split}
	f(\rho) &= \sup_{\mu} \bigl( \mu \rho - p(\mu)\bigr) \\
	&= f(\rho_{\rm sat})+  \frac{1}{2\sigma^2} ( \rho- \rho_{\rm sat})^2 - \sum_{k\geq 0} \lambda_k ( \rho- \rho_{\rm sat})^{3+k},
	\end{split}
\ee
as $\rho \nearrow \rho_{\rm sat}$.
The variance $\sigma^2$ is proportional to the isothermal compressibility $\kappa_{\sf T}$ of the saturated gas~\cite{sator}: A formal manipulation in thermodynamics notation shows that
\be
	\frac{1}{\kappa_{\sf T}} = -  V\frac{\partial p}{\partial V} \Big|_{T} = V \frac{\partial^2}{\partial V^2} \bigl\{V f \bigl(\tfrac{N}{V}\bigr)\bigr\} = \rho^2 \frac{\partial^2 f}{\partial \rho^2}(\rho) 
\ee
which converges to $\frac{\rho_{\rm sat}^2}{\sigma^2}$ as $\rho\nearrow \rho_{\rm sat}$. Here, $V$ is the volume, $N=\rho V$ the particle number, and $F = V f(N/V)$ the extensive free energy.
 Eq.~\eqref{eq:surface} is up to an additive constant  the free energy of  supersaturated gas in coexistence with a large droplet; the shift $\Delta_L \propto L^{d-1} \propto n^{(d-1)/d}$ is natural because surface tension tends to make the condensate droplet smaller.  This gives a physical interpretation to the large deviations for stretched exponential variables with $\gamma  =(d-1)/d \geq 1/2$. 

The anomalous fluctuations arise when the compressibility $\kappa_{\sf T} \propto \sigma^2$ becomes infinite, i.e., at the critical point of the fluid. In the  Fisher droplet model, this is achieved by  assigning to each droplet a ``surface'' free energy more general than $j^{(d-1)/d}$, with temperature-dependent parameters; see the review~\cite{sator} and the references therein. 

\begin{proof} [Proof of Proposition~\ref{prop:equivalence}]
	Observe that for $k\geq0$,
	\be
		\bbP\bigl( \sum_{j\geq 1} j N_j = k\bigr) = h_k \e{-\sum_{j\geq 1} \theta_j /j }
	\ee
	(with $h_0=1$).
	By Eq.~\eqref{eq:levy} the right side adds up to $1$ when $k$ is summed over $k\geq 0$, so that $\sum_{j\geq 1} j N_j$ is almost surely finite. 
	Let $S_n$ be the size of the largest partition element, at the site of higher occupation. 
	Clearly $S_L \leq T_L \leq M_L$; we show below that $S_L = T_L$ with high probability.
	Fix $k \in \bbN_0$ and let $\delta>0$ with $\delta <1- \rho_{\rm c}/\rho$. Choose $n$ large enough so that $\delta n - k > k$. 
 Using Proposition~\ref{prop:conditional}, we get 
	\be
		\bbP_{L,n} \bigl( S_L = M_L - k\bigr)  = \sum_{m\geq \delta n} \frac{\theta_{m-k}}{m-k}  \frac{h_k}{h_m} \bbP_{L,n}\bigl( M_L = m\bigr) +  \bbP_{L,n} \bigl( S_L = M_L - k, M_L \leq \delta n\bigr),
	\ee
	hence 
	\begin{multline}
		\Bigl| \bbP_{L,n} \bigl( S_L = M_L - k\bigr) - \bbP\bigl(\sum_{j\geq 1} j N_j =k \bigr) \Big| \\
		\leq 2 \bbP_{L,n}\bigl( M_L \leq \delta n) + \sup_{m \geq \delta n} \Bigl| \frac{\theta_{m-k}}{m-k}  \frac{h_k}{h_m} - h_k \exp\Bigl(-\sum_{j\geq 1} \frac{\theta_j}{j} \Bigr) \Bigr|.
	\end{multline}
	The first term vanishes because $\frac{1}{n} M_L \to 1- \frac{\rho_{\rm c}}{\rho} >\delta$; the second term vanishes because $(\theta_j/j)$ is discrete subexponential (Theorems~\ref{thm EH} and~\ref{thm buv}). Thus we have shown that for all $k\geq0$,
	\be \label{eq:snlimit}
		\limtwo{n\to\infty}{\rho L^d = n} \bbP_{L,n} \bigl( S_L = M_L - k\bigr) = \bbP\bigl(\sum_{j\geq 1} j N_j =k \bigr),
	\ee
	that is, $M_L - S_L$ converges in law to $\sum_{j\geq 1} j N_j$. It follows in particular  that $\frac{1}{n} S_n \to 1- \frac{\rho_{\rm c}}{\rho}$. 
To conclude, we note that 
	since $S_L \leq T_L$, we have 
	\be
		\bbP_{L,n}\bigl(  T_L\neq S_L)  \leq  \bbP_{L,n}\bigl(  T_L>S_L \mid S_L \geq \delta n \bigr) + \bbP_{L,n} \bigl( S_L <\delta n\bigr).
	\ee
	The last term vanishes because $\frac{1}{n} S_L \to 1- \frac{\rho_{\rm c}}{\rho}$ and $\delta < 1-\rho_{\rm c}/\rho$. The first term in the right side is bounded by the conditional probability that there is a second site, in addition to the most occupied one, whose occupation is at least $\delta n$.  With the help of Eqs.~\eqref{eq:nagaev} and~\eqref{eq:heavy-tailed} one can show that the probability of the latter event goes to zero. Hence $\bbP_{L,n}(T_L \neq S_L) \to 0$ and we conclude with the help of Eq.~\eqref{eq:snlimit}.
\end{proof} 

In forthcoming work \cite{ercolani-jansen-ueltschi} we will introduce analytic continuations, known as {\it Lindel\"of integrals}, of the grand canonical ensembles studied in this paper and, by methods of Hardy space analysis and contour deformation, show how the results of this section may be extended to a broader class of decaying weights in which the weights $\theta_j$ arise as the values, at positive integers $j$, of a function $p(\xi)$ analytic for $\xi \in \mathbb{C} \setminus (-\infty, 0]$.

\medskip
{\footnotesize
{\bf Acknowledgments:}
We thank the Laboratoire Jean Dieudonn\'e of the University of Nice, and the Institut Henri Poincar\'e (during the program of the spring 2013 organized by M. Esteban and M. Lewin) for kindly hosting us. S.J. is grateful for helpful discussions on random partitions and heavy-tailed variables with J. Blath and M. Scheutzow, and we acknowledge useful discussions with S. Grosskinsky on the zero-range process. The referees made useful comments.
The authors gratefully acknowledge the support of NSF grant DMS-1212167 (N.M.E.),  EPSRC grant EP/G056390/1 (D.U.), and ERC advanced grant 267356 VARIS  of F. den Hollander (S.J.).

\end{document}